\documentclass[10pt,final]{siamltexmm}
\usepackage{color,amsmath,amssymb,mathrsfs}
\usepackage{graphicx}
\usepackage{algorithmic,algorithm}
\usepackage{tikz,pgfplots}
\usepackage{upgreek}
\usepackage{showkeys,cite}
\usepackage{mathtools}
\usepackage[normalem]{ulem}
\DeclareMathAlphabet{\mathpzc}{OT1}{pzc}{m}{it}
\usepackage{tikz} 

\usepackage{mathtools}
\definecolor{grassgreen}{RGB}{92,135,39}
\newcommand{\hilb}{\mathscr{H}}

\renewcommand{\vec}[1]{{\mathchoice
                     {\mbox{\boldmath$\displaystyle{#1}$}}
                     {\mbox{\boldmath$\textstyle{#1}$}}
                     {\mbox{\boldmath$\scriptstyle{#1}$}}
                     {\mbox{\boldmath$\scriptscriptstyle{#1}$}}}}

\newcommand{\trace}{\mathrm{Tr}}
\newcommand{\eps}{\varepsilon}
\newcommand{\euclidnorm}[1]{\left\| {#1} \right\|_2}
\newcommand{\norm}[1]{\left\| {#1} \right\|}
\newcommand{\ipbig}[2]{{\big\langle {#1}, {#2} \big\rangle}}
\newcommand{\ip}[2]{{\left\langle {#1}, {#2} \right\rangle}}

\newcommand\restr[2]{{
  \left.\kern-\nulldelimiterspace 
  {#1}\vphantom{\big|} \right|_{#2}}}

\newcommand{\R}{\mathbb{R}}
\newcommand{\Z}{\mathbb{Z}}

\newcommand{\A}{\mathcal{A}}

\newcommand{\C}{\mathcal{C}}
\newcommand{\D}{\mathcal{D}}

\newcommand{\F}{\mathcal{F}}
\newcommand{\G}{\mathcal{G}}
\newcommand{\J}{\mathcal{J}}

\newcommand{\Q}{\mathcal{Q}}

\newcommand{\borel}{\mathscr{B}}
\newcommand{\ave}{\operatorname{E}}
\newcommand{\var}{\operatorname{Var}}
\newcommand{\GM}[2]{\mathcal{N}\!\left( {#1}, {#2}\right)}

\newcommand*\oline[1]{%
  \vbox{%
    \hrule height 0.6pt
    \kern0.25ex
    \hbox{%
      \kern-0.15em
      \ifmmode#1\else\ensuremath{#1}\fi
      \kern-0.15em
    }
  }
}

\newcommand{\obsq}{\bar{\vec{\mathrm q}}}
\newcommand{\ipar}{m}
\newcommand{\iparb}{{\bar{\ipar}}}

\newcommand{\Ntr}{{n_{\text{tr}}}}

\renewcommand{\H}{\mathcal{H}}

\renewcommand{\L}{\mathscr{L}}

\newcommand{\CM}{\mathscr{E}}

\newcommand{\cip}[2]{\left\langle{#1}, {#2}\right\rangle_{\!\CM}}

\newcommand{\ad}[1]{{#1}^\star}

\newcommand{\ut}[1]{\ensuremath{\tilde{#1}}}

\newcommand{\Exp}[1]{e^{{#1}}}

\newcommand{\commentout}[1]{\iffalse {#1} \fi}

\newcommand{\Vg}{\mathscr{V}_{\!g}}
\newcommand{\V}{\mathscr{V}}

\renewcommand{\O}{\mathcal{O}}
\newcommand{\QoI}{\Theta}
\newcommand{\ctrl}{z}
\newcommand{\Nc}{{n_\text{c}}}

\allowdisplaybreaks[3]

\DeclareMathAlphabet      {\mathup}{OT1}{\familydefault}{m}{n}

\renewcommand{\beta}{\upbeta}
\newcommand{\precond}[1]{\C^{1/2}{#1}\C^{1/2}}
\newcommand{\GD}{\ensuremath{\Gamma_{\!\!D}}}
\newcommand{\GN}{\ensuremath{\Gamma_{\!\!N}}}
\def\qed{\rule{1.5ex}{1.5ex}}
\newcommand{\Grad}{{\QoI_\ipar}}
\newcommand{\Hess}{\QoI_{\ipar\ipar}}
\newcommand{\Zad}{\mathscr{Z}_\mathup{ad}}
\newcommand{\K}{\mathcal{K}}
\newcommand{\proj}[2]{\sum_{{#2}=1}^n \ip{{#1}}{e_{{#2}}}e_{{#2}} }

\newcommand{\vara}{a}
\newcommand{\varb}{b}

\renewcommand{\slugline}[2]{}
\renewcommand{\slugger}[2]{}

\begin{document}

\def\addressncsu{Department of Mathematics, North Carolina State University,
  Raleigh, NC, USA}
\def\addressnyu{Courant Institute of Mathematical Sciences, New York
  University, New York, NY, USA}
\def\addressucm{Applied Mathematics, School of Natural Sciences, University of California,
  Merced, CA, USA}
\def\addressomar{Institute for Computational Engineering \& Sciences,
  Department of Mechanical Engineering, and Department of Geological Sciences, The
  University of Texas at Austin, Austin, TX, USA}

\author{Alen Alexanderian\footnotemark[1]
  \and Noemi Petra\footnotemark[2]
  \and Georg Stadler\footnotemark[3]
  \and Omar~Ghattas{\footnotemark[4]}}
\renewcommand{\thefootnote}{\fnsymbol{footnote}}
\footnotetext[1]{\addressncsu\ (\email{alexanderian@ncsu.edu}).} 
\footnotetext[2]{\addressucm\ (\email{npetra@ucmerced.edu}).}
\footnotetext[3]{\addressnyu\ (\email{stadler@cims.nyu.edu}).}
\footnotetext[4]{\addressomar\ (\email{omar@ices.utexas.edu}).} 

\renewcommand{\thefootnote}{\arabic{footnote}}

\title{
Mean-variance risk-averse optimal control of systems governed by PDEs with random parameter 
fields using quadratic approximations\thanks{This work was partially
  supported by
NSF grants 1508713 and 1507009, %
and DOE grants
DE-FC02-13ER26128, %
DE-SC0010518, %
and DE-FC02-11ER26052.}} %

\maketitle
\newcommand{\slugmaster}{%
\slugger{juq}{}{}{}{}}%

\begin{abstract}
We present a method for optimal control of systems governed by partial
differential equations (PDEs) with uncertain parameter fields.  We consider an
objective function that involves the mean and variance of the control
objective,
leading to a risk-averse optimal control problem.  Conventional numerical
methods for optimization under uncertainty are prohibitive when applied to this
problem.
To make the optimal control problem tractable, we invoke a 
quadratic Taylor series approximation of the control objective with
respect to the uncertain parameter field. This enables deriving explicit
expressions for the mean and variance of the control objective in terms of its
gradients and Hessians with respect to the uncertain parameter.  
The risk-averse optimal control problem is then formulated as a PDE-constrained optimization
problem with constraints given by the forward and adjoint PDEs defining  
these gradients and Hessians.  The expressions
for the mean and variance of the control objective under the quadratic
approximation involve the trace of the (preconditioned) Hessian, and are thus
prohibitive to evaluate. To overcome this difficulty, we employ
trace estimators, which only require a modest number of Hessian-vector
products.
We illustrate our approach with two specific problems: the control of a
semilinear elliptic PDE with an uncertain boundary source term, and the control of a
linear elliptic PDE with an uncertain coefficient field. For the latter
problem, we derive adjoint-based expressions for efficient computation of the
gradient of the risk-averse objective with respect to the controls. Along with
the quadratic approximation and trace estimation, this ensures that the cost of
computing the risk-averse objective and its gradient with respect to the
control---measured in the number of PDE solves---is independent of the
(discretized) parameter and control dimensions, and depends only on the number
of random vectors employed in the trace estimation, leading to an efficient
quasi-Newton method for solving the optimal control problem.
Finally, we present a comprehensive numerical study of an optimal control
problem for fluid flow in a porous medium with uncertain permeability field. 
\end{abstract}



\section{Introduction}
An important class of problems arising in engineering and science is
the optimization or optimal control of natural or engineered systems
governed by partial differential equations (PDEs).  Often,
the PDE models of these systems are characterized by
parameters (or parameter functions) that are not known and are 
considered uncertain and modeled as random variables.
These parameters can appear for example as coefficients, boundary
data, initial conditions, or source terms.  Consequently, 
optimization of such systems should be done in a way that the computed
optimal controls or designs are robust with respect to the variability
in the uncertain parameters.

\paragraph{Literature survey and challenges}  There is a rich
body of literature on theoretical and computational aspects of optimal 
control of systems governed by PDEs~\cite{Troltzsch10, Gunzburger03,
  BieglerGhattasHeinkenschlossEtAl07, BorziSchulzSchillingsEtAl10,
  HinzePinnauUlbrichEtAl09}, and of optimization under
uncertainty (OUU)~\cite{BirgeLouveaux97,Sahinidis04, BeyerSendhoff2007,ShapiroDentchevaRuszczynski09}.
Recently there has been considerable interest in solution methods for
optimization problems lying at the intersection of these
two fields, namely, optimization problems governed by PDEs with
uncertain parameters~\cite{BorzivonWinckel09,
BorziSchulzSchillingsEtAl10,
BorziVonWinckel11, 
GunzburgerMing11,
HouLeeManouzi11,
TieslerKirbyXiuEtAl12, 
Kouri12,
KouriHeinkenschlossRidzalEtAl13, 
KouriHeinkenschlossRidzalEtAl14,
KunothSchwab13,
Kouri14,
ChenQuarteroni14,
KouriSurowiec16,
DambrineDapognyHarbrecht15
}.
To discuss the challenges of 
optimal control, and more generally optimization, of systems governed
by PDEs with uncertain parameters, we consider a real-valued
optimization objective $\QoI(\ctrl, \ipar)$ that depends on a control
variable $\ctrl$ and an uncertain parameter $\ipar$, both of which can be
finite- or infinite-dimensional.
Throughout this
article we refer to $\QoI(\ctrl, \ipar)$ as the control objective.
The evaluation of this control objective requires the solution of a system of PDEs.
Namely,  $\QoI(\ctrl, \ipar) :=
\tilde\QoI(\ctrl,\ipar,u)$ with $u = \mathcal S(\ctrl,\ipar)$, 
where $\mathcal{S}$ is a  
PDE solution operator. 
Here, we assumed that the system
of PDEs admits a unique solution $u$ for every pair $(\ctrl,\ipar)$ of
controls and parameters.
This dependence of $\QoI$ on the solution of a PDE makes the evaluation (and
the computation of derivatives) of $\QoI$ computationally expensive. 
The presence of uncertain parameters greatly compounds the
computational challenges of solving the PDE-constrained optimization
problem.

In an OUU problem, it is natural to seek optimal controls $\ctrl$ that
make $\Theta$ small in an average sense.  For example, a risk-neutral
optimal control approach seeks controls that solve 
\begin{equation}\label{equ:risk-neutral}
    \min_\ctrl \, \ave\{\QoI(\ctrl, \ipar)\}, 
\end{equation}
where $\ave\{\cdot\}$ denotes expectation over the uncertain parameter
$\ipar$.
If we seek controls that, in addition to 
minimizing the expected value of $\QoI$ with respect to $\ipar$,  
result in a small uncertainty in $\QoI$, we are led to
risk-averse optimal control.
In the present work, we use the variance of the control objective
as a risk measure, and seek optimal controls that solve the problem 
\begin{equation}\label{equ:risk-averse-general}
   \min_\ctrl \, \ave\{\QoI(\ctrl, \ipar)\} + \beta \var\{ \QoI(\ctrl, \ipar)\}.
\end{equation}
Here, $\var\{ \cdot\}$ denotes the variance with respect to $\ipar$,
and $\beta > 0$ is a \emph{risk-aversion} parameter that aims to
penalize
large variances of the control objective.  
This {\em mean-variance} formulation is
only one of several formulations for finding risk-averse optimal controls.
Other examples of more complex risk measures include the value at risk (VaR) and the
conditional value at risk
(CVaR)~\cite{RockafellarUryasev00,ShapiroDentchevaRuszczynski09}.
Compared to the mean-variance formulation, these
approaches do not symmetrically penalize the deviation of the quantity
of interest 
around the mean, which is desirable for instance in applications where the
control objective models a loss.
In~\cite{KouriSurowiec16}, 
the authors consider primal and
dual formulations of a risk-averse PDE-constrained OUU problem using
CVaR. 
They employ and study smooth
approximations of the primal formulation to enable the
application of derivative-based optimization methods to the CVaR
objective, and rely on quadrature-based discretizations in the
discretized %
parameter space.

To illustrate the main computational challenges involved in OUU
problems, let us consider the (simpler) risk-neutral problem.  A
common approach to cope with the expectation in the objective
function uses sampling over the random parameter space, 
$\ave\{\QoI(\ctrl, \ipar)\} \approx \sum_{i = 1}^n w_i \QoI(\ctrl,
\ipar_i)$, where $\{\ipar_i\}_{i = 1}^n$ is a sample set, and $w_i$
are sample weights.  In the context of PDE-constrained OUU,
evaluation of $\QoI(\ctrl, \ipar_i)$ requires solving the PDE
problem $u_i = \mathcal S(\ctrl,\ipar_i)$ for each sample point
$\ipar_i$.  The sample set $\{\ipar_i\}_{i = 1}^n$ is chosen either by
Monte Carlo sampling, where each $\ipar_i$ is a draw from the
distribution law of $\ipar$ (and $w_i = 1/n$ for every $i$), or, for a
suitably low-dimensional parameter space, 
based on quadrature rules (and $w_i$ are quadrature weights). The Monte
Carlo-based approach, sometimes referred to as sample average
approximation (SAA), is computationally prohibitive for OUU problems
governed by PDEs. This is due to the slow convergence of Monte Carlo
and the resulting large number of PDE solves for each evaluation of
the expectation.
Quadrature-based methods, obtained from tensorization of one-dimensional
quadrature rules, use regularity of $\QoI(\ctrl,\ipar)$ with respect to $\ipar$
and can accelerate convergence, but they suffer from the the curse of
dimensionality---the exponential growth of the number of quadrature points as
the dimension increases.  The use of sparse quadrature~\cite{Smolyak63} can
mitigate but not overcome the curse of dimensionality. Quadrature-based
methods can be improved significantly by using adaptive sparse grids (see
e.g.,~\cite{KouriHeinkenschlossRidzalEtAl13,BorziSchulzSchillingsEtAl10,
KouriHeinkenschlossRidzalEtAl14}, in which adaptive sparse grids are employed
to solve OUU problems); however these approaches are still computationally
expensive for problems with parameter dimensions in the order of
hundreds or thousands. 
Another class of methods for OUU problems are stochastic approximation (SA)
methods~\cite{RobbinsMonro51,Gaivoronskii78,Ermoliev83,RuszczynskiWojciech86}.
Similar to methods based on SAA, SA methods are computationally intractable for
PDE-constrained OUU problems with high-dimensional parameters due to their slow
convergence and the resulting need for a prohibitively large number of PDE
solves.

\paragraph{Approach}
We consider an uncertain parameter $\ipar$ that is modeled with a random field,
which can also be viewed as a function-valued random variable.
In the uncertainty quantification literature, it is a common to use an
a priori dimension reduction provided by a truncated Karhunen--Lo\`{e}ve (KL)
decomposition for such problems.
However, KL modes that appear unimportant in simulating the random
process may turn out to be important to the control objective.  Moreover, a
priori truncation of the KL expansion of a random field is most useful if the
eigenvalues of the covariance operator (of the uncertain parameter) exhibit
rapid decay. This is not always the case, for example, in presence of
small correlation lengths; in such cases, an a priori truncation needs to 
retain a large number of KL modes.
We do not follow such approaches to avoid bias introduced by
the truncated KL expansion.  Instead, we seek formulations that
preserve the problem's infinite-dimensional character and work in an
infinite-dimensional setting as long as possible. Moreover, we aim to
devise algorithms whose computational complexity, measured in the
number of PDE solves, is independent of the discretized parameter
dimension.

In the present work, we employ quadratic approximations of the
parameter-to-objective map, $\ipar \mapsto \QoI(\cdot, \ipar)$, to
render the computation of the control objective and its gradient (as
necessitated by a gradient-based optimization method) tractable.  
Related approaches for OUU with finite-dimensional uncertain parameters and 
inexpensive-to-evaluate (compared to problems governed by PDEs)
control objectives are used
in~\cite{DarlingtonPantelidesRustemEtAl99,DarlingtonPantelidesRustemEtAl00}.
More generally, linear or quadratic expansions with respect to
uncertain finite-dimensional parameters have also been used for robust
(finite-dimensional) optimization and reliability methods in
engineering applications; we refer, e.g., to
\cite{DiehlBockKostina06,Rackwitz01}.
Using this approach, we can compute the moments of the first and
second-order Taylor expansions of
$\QoI(\ctrl, \ipar)$ analytically.  For optimal
control problems with infinite-dimensional parameters, computation
of the derivatives with respect to the uncertain parameters and the
controls is prohibitive using the direct sensitivity approach (or
finite differences). Instead, we employ adjoint methods to avoid
dependence on the dimension of the discretized parameter field.
Our formulation is particularized to two model problems, the control of a
semilinear elliptic PDE with an uncertain boundary source term, and the
control of a linear elliptic PDE with an uncertain coefficient field.
The latter is motivated by industrial problems
involving the optimal control of flows in porous media. 

As we will see, using the quadratic approximation of $\ipar \mapsto
\QoI(\cdot, \ipar)$ 
results in an OUU objective function that involves traces of
operators that  depend on the Hessian of this mapping.
Since direct computation of these traces is 
prohibitive for high-dimensional problems 
(explicit computation of the Hessian requires as many PDE solves as
there are parameters), we use trace estimation either based on random
vectors, or on eigenvectors 
of the (preconditioned) Hessian at a nominal control. This
only requires the action of the Hessian on vectors and is thus well
suited for control problems governed by systems of PDEs.
\paragraph{Contributions}
The main contributions of this work are as follows: (1) For an uncertain
parameter field that follows a Gaussian distribution law, we derive analytic
expressions for the mean and variance of a quadratic approximation to the
parameter-to-objective map in infinite dimensions.  These results are the basis
for developing an efficient OUU approach that extends the work
in~\cite{DarlingtonPantelidesRustemEtAl00} to a method suitable for
large-scale PDE-constrained OUU problems.
(2) We propose a formulation of the risk-averse OUU problem as a
PDE-constrained optimization problem, 
with the constraints given by the
PDEs defining the adjoint-based expression for the gradient $\Grad$
and the linear action of the Hessian $\Hess$ of the parameter-to-objective 
map. 
Our method ensures that the cost of computing the risk-averse objective and its
gradient with respect to the control---measured in the number of PDE
solves---is independent of the (discretized) parameter and control dimensions,
and depends only on the number of random vectors used in the trace
estimation. 
(3) 
We fully elaborate our approach for the risk-averse control of an elliptic PDE
with uncertain coefficient field, and in particular derive the adjoint-based
gradient of the control objective with respect to the control.
We numerically
study various aspects of our
risk-aversion measure and of the efficiency of the method. 
The results show the effectiveness of our approach
in computing risk-averse optimal controls in a problem with a $3{,}000$-dimensional
discretized parameter space.

\paragraph{Limitations} 
We also remark on limitations of our method.
(1) Since we rely on
approximations based on Taylor expansions, our arguments require
smoothness (and proper
boundedness of the derivatives) of the parameter-to-objective map.
For the mean-variance
formulation used here, this smoothness mainly depends on the
governing PDEs and how the parameter enters these PDEs. 
(2)
Compared to sampling-based methods, our approach requires first and
second derivatives of $\QoI(\ctrl, \ipar)$ with respect to
$\ipar$, and thus appropriate adjoint solvers for the efficient
computation of these derivatives.
However, efficient methods for solution of optimal control problems
governed by PDEs will require adjoint-based derivatives with respect
to the control; only minor modifications are needed to obtain
derivatives with respect to the uncertain parameters.
(3) Our derivation of the mean and variance of the quadratic
approximation to the parameter-to-objective map assumes that the
parameter space is a Hilbert space, say $\hilb$, and that $\QoI(\ctrl,
\cdot)$ is defined and has the required derivatives in $\hilb$.  As
such, our framework does not apply to cases where the parameter space
is a general Banach space.  Even when we consider a Hilbert space
$\hilb$ of uncertain parameters, $\QoI(\ctrl, \ipar)$ may not be
defined for every $\ipar$ in $\hilb$ and have derivatives that are
bounded in $\hilb$, in which case the expressions for the mean and
variance could be used only formally to obtain an objective function
for a risk-averse OUU problem.  This is the case for the linear
elliptic PDE problem with uncertain coefficient field, where the
parameter-to-objective map is defined and has the required smoothness
only in a subspace
that has full measure. In section~\ref{sec:extensions}, we discuss such issues further and give conditions that
ensure that the expressions for the mean and variance are well
defined.
Extensions to more general cases is a subject of our future
work. 

\section{Preliminaries}
We let $\hilb$ be an infinite-dimensional real separable Hilbert space endowed with an
inner product $\ip{\cdot\,}{\cdot}$ and induced norm $\norm{\cdot}^2 =\ip{\cdot\,}{\cdot}$. 
We consider uncertain parameters that are modeled as 
spatially distributed random processes, which can be viewed as 
function-valued random variables. 
Let $\ipar$ denote such an uncertain parameter. 
We assume the distribution law of $\ipar$, which we denote by $\mu$, is supported on $\hilb$ and consider 
$\ipar$ as an $\hilb$-valued random variable. 
That is, $\ipar$ is a function, $\ipar: (\Omega, \Sigma, P) \to
(\hilb, \borel(\hilb))$, where $\Omega$ is a sample space, $\Sigma$ is an
appropriate sigma-algebra, and $P$ is a probability measure; here,
$\borel(\hilb)$ denotes the Borel sigma-algebra on $\hilb$. 
In what follows, with a slight abuse of notation, 
we denote the realizations of the uncertain parameter using the
same symbol $\ipar$.  

As is common practice, instead of working on the abstract probability
space $(\Omega, \F,
P)$, we work in the image space $(\hilb, \borel(\hilb), \mu)$,
where $\mu$ is the law of $\ipar$, and use 
$\ave\{ \phi(\ipar) \} = \int_\hilb \phi(m) \mu(dm)$ for an integrable function $\phi:\hilb \to \R$.
As mentioned in the introduction, we assume that $\ipar$ has a Gaussian
probability law, $\mu = \GM{\iparb}{\C}$. Here, $\C$ is a
self-adjoint, positive trace class operator and thus $\mu$ defines a Gaussian measure on $\hilb$. 

We denote by $\Zad$ the set of admissible controls, which is a closed convex subset
of $L^2(\D)$, where $\D \subset \R^d$ is a bounded open set with piecewise smooth boundary.
We consider the control objective, $\QoI$ as a real-valued function defined on $\Zad \times \hilb$.
We refer to the mapping $\ipar \mapsto \QoI(\cdot,\ipar)$ as the parameter-to-objective map.
For normed linear spaces $\mathscr{X}$ and $\mathscr{Y}$ we denote by $\mathcal{L}(\mathscr{X}, \mathscr{Y})$ 
the space of bounded linear transformations from $\mathscr{X}$ to $\mathscr{Y}$, and by $\mathcal{L}(\mathscr{X})$ the space of
bounded linear operators on $\mathscr{X}$. For a Hilbert space $\hilb$, we use $\mathcal{L}_\mathup{sym}(\hilb)$ 
to denote the subspace of $\mathcal{L}(\hilb)$ consisting of self-adjoint linear operators. 

\subsection{Linear and quadratic expansions}

Recall that for a function $f:\mathscr{X} \to \R$, where
$\mathscr{X}$ is a Banach space, existence of first and second
Fr\'{e}chet derivatives at $\iparb \in \mathscr{X}$ implies that
\begin{equation}\label{eq:Frechet}
   f(\ipar) = f(\iparb) + f'(\iparb)[\ipar - \iparb] + \frac12 f''(\iparb)[\ipar-\iparb, \ipar-\iparb] + 
   o(\norm{\ipar - \iparb}_\mathscr{X}^2),
\end{equation}
with $f'(\iparb) \in \mathscr{X}^*$ and $f''(\iparb) \in \mathcal{L}(\mathscr{X}, \mathscr{X}^*)$.
In the following, we consider the case where $\mathscr{X}$ is a Hilbert space $\hilb$, and hence the derivatives 
admit Riesz representers in $\hilb$. 

Next, we consider the control objective, $\Theta:\Zad \times\hilb \to
\R$ and assume for an arbitrary control $z \in \Zad$, 
existence of first and second
Fr\'{e}chet derivatives for 
the parameter-to-objective map $\ipar \mapsto \Theta(\ctrl, \ipar)$ 
at $\iparb \in \hilb$. 
Consider the linear and quadratic approximations of the parameter-to-objective map: 
\begin{align}
\QoI_\mathup{lin}(\ctrl, \ipar) &= \QoI(\ctrl, \iparb) + \ip{\Grad(\ctrl,
\iparb)}{\ipar - \iparb},\label{equ:linearapprox}\\
\QoI_\mathup{quad}(\ctrl,\ipar) &= \QoI(\ctrl,\iparb) +
\ip{\Grad(\ctrl,\iparb)}{\ipar - \iparb} + \frac12 \ip{\Hess(\ctrl,
\iparb)(\ipar - \iparb)}{\ipar-\iparb}. \label{equ:quadapprox}
\end{align}
Notice that for clarity, we denote first and second derivatives of $\QoI$ with respect
to $\ipar$ by $\Grad$ and $\Hess$, respectively.
In this paper, we mainly use the quadratic approximation $\QoI_\mathup{quad}$.

Using an approximation rather than the exact parameter-to-objective map
$\QoI(\ctrl, \ipar)$ enables computing the moments appearing in the
objective function of the (approximate) risk-averse OUU problem analytically.  This
computation is facilitated
by the fact that $\ipar$ has a Gaussian distribution law.  
Note also that the accuracy of such linear and quadratic approximations, considered in an 
average sense, can be related to the variance of the uncertain parameter. 
In section~\ref{sec:analytic}, where we derive analytic expressions for the mean and
variance of the local quadratic approximation to $\QoI(\ctrl, \ipar)$, we also
describe how the variance of $\ipar$ is related to the expected
value of the truncation error in the quadratic approximation.

\subsection{Probability measures on $\hilb$}\label{sec:measures}
Here we recall basics regarding Borel probability measures, and
Gaussian measures on infinite-dimensional Hilbert spaces.
Let $\mu$ be a Borel probability measure on $\hilb$ with finite first and 
second moments. The mean $\bar{\vara}$ of $\mu$ is an element of $\hilb$ such that,
\[
    \int_\hilb \ip{s}{\varb}\,\mu(ds) = \ip{\bar{a}}{\varb} \quad \text{for all } \varb \in \hilb.
\]
The covariance operator $\C$ of $\mu$ is a positive self-adjoint trace-class operator 
that satisfies
\[
   \int_\hilb \ip{\vara}{s-\bar{\vara}}\ip{\varb}{s-\bar{\vara}} \, \mu(ds) = \ip{\C \vara}{\varb} \quad
   \text{for all } \vara, \varb \in \hilb.
\]
It is straightforward to show that
$\int_\hilb \norm{s-\bar{\vara}}^2 \, \mu(ds) = \trace(\C)$, where $\trace(\C)$ denotes the trace of the
(positive self-adjoint) operator $\C$; see, e.g.,~\cite[p.~8]{Prato06}.
Note also that 
\begin{equation}\label{equ:Cvv}
\int_\hilb \ip{\varb}{s-\bar{\vara}}^2 \, \mu(ds) 
\\
= \ip{\C \varb}{\varb}.
\end{equation}
One can also show (see e.g.,\cite[Lemma 1]{AlexanderianGloorGhattas16})
that for a bounded linear operator $\K: \hilb \to \hilb$, 
\begin{equation}\label{equ:quadform_1st_moment}
   \int_\hilb \ip{\K(s - \bar{\vara})}{s - \bar{\vara}}\,\mu(ds) = \trace(\K \C) = \trace(\C^{1/2} \K \C^{1/2}).
\end{equation}

Now consider the case where the measure $\mu$ is a Gaussian, $\mu =
\GM{\bar{\vara}}{\C}$, where $\C$ is a positive, self-adjoint, trace-class
operator. The mean $\bar{\vara}$ is assumed to belong to the Cameron--Martin
space $\CM:=\text{Im}(\C^{1/2})$. 
The Cameron--Martin space is a dense subspace of
$\hilb$ and is a Hilbert space endowed with the inner product
$\ip{\cdot\,}{\cdot}_{\CM}:=\ip{\C^{-1/2}\cdot}{\C^{-1/2}\cdot}$~\cite{DaPratoZabczyk02}.
While the space $\CM$ is dense in $\hilb$, it is in some sense very ``thin''; more
precisely, $\mu(\CM) = 0$. 

We follow the construction
in~\cite{Bui-ThanhGhattasMartinEtAl13}, and define a Gaussian measure $\mu$ 
on $\hilb = L^2(\D)$, where $\D$ is a bounded domain with piecewise smooth boundary, as follows: 
define the covariance operator as the inverse of the square of a Laplacian-like operator:
\begin{equation}\label{eq:Cconcrete}
   \C = (-\kappa \Delta + \alpha I)^{-2} =: \mathcal{A}^{-2}, 
\end{equation}
where $\kappa,\alpha>0$, and the domain of $\A$ is given by 
$D(\A) = \left\{ u \in H^2(\D) : \nabla u \cdot \vec{n} = 0 \text{ on }\partial\D \right\}$. 
Here, $H^2(\D)$ is the Sobolev space of $L^2(\D)$ functions with square
integrable first and second weak derivatives, and $\vec{n}$ is the unit
outward normal for the boundary $\partial\D$. This construction of the operator 
$\C$ 
ensures that it is positive, self-adjoint, and of trace-class and thus 
the Gaussian measure $\mu = \GM{\bar{a}}{\C}$
on $\hilb$ is well-defined.
A Gaussian random field whose law is given by such a Gaussian measure 
has almost surely continuous realizations~\cite{Stuart10}.

As we saw in~\eqref{equ:quadform_1st_moment}, given a Borel probability measure on
$\hilb$ with bounded first and second moments, 
we can obtain a simple expression for the first moment of a quadratic
form on $\hilb$.  If $\mu$ is a Gaussian measure $\GM{\bar{\vara}}{\C}$ on
$\hilb$, we can also compute the second moment of a quadratic form (see Remark
1.2.9.\ in~\cite{DaPratoZabczyk02}). In particular, if we let $\K$ be a
self-adjoint bounded linear operator on $\hilb$, then 
\begin{equation}\label{equ:quadform_2nd_moment}
\int_\hilb \ip{\K (s - \bar{\vara})}{s - \bar{\vara}}^2 \, \mu(ds) = 2 \trace\big[(\precond{\K})^2\big] + \trace(\precond{\K})^2. 
\end{equation}

\section{Risk-averse OUU with quadratic approximation of the
parameter-to-objective map}\label{sec:analytic} As discussed above, we
consider a control objective
$\QoI = \QoI(\ctrl, \ipar)$, where $\ipar$ has a Gaussian
distribution law $\mu = \GM{\iparb}{\C}$. In
section~\ref{sec:moments_of_quad_approx}, we analytically derive
the moments of the quadratic approximation $\QoI_\mathup{quad}$
of $\QoI$.
We discuss
the approximation errors due to this approximation by
studying the expected value of the remainder 
term in the Taylor expansion.  In section~\ref{sec:ouu_objective}, using the expressions 
for the  moments of the quadratic approximation, we 
formulate the optimization problem for finding risk-averse optimal controls.
Extensions of our OUU approach to problems where
$\QoI(\ctrl, \ipar)$ is defined only in a subspace of $\hilb$ are discussed in
section~\ref{sec:extensions}.

\subsection{Quadratic approximation to a function of a Gaussian random variable}
\label{sec:moments_of_quad_approx}
In this section, we compute mean and variance of $\QoI_\mathup{quad}$
defined in~\eqref{equ:quadapprox} in the infinite-dimensional Hilbert space setting.
The following arguments are pointwise in the control $\ctrl \in \Zad$ and hence,
for notational convenience, we suppress the dependence of $\QoI$ on $\ctrl$.
We begin by establishing the following technical result:
\begin{lemma}\label{lem:basic}
Let $\mu = \GM{\bar{\vara}}{\mathcal{C}}$ be a Gaussian measure on $\hilb$, and $\varb \in \hilb$ be fixed, 
and let $\K$ be a bounded linear operator on $\hilb$. Then,
\[\int_\hilb \ip{\varb}{s-\bar{\vara}}\ip{\K (s - \bar{\vara})}{s - \bar{\vara}} \mu(ds) = 0.\]
\end{lemma}
\begin{proof}
Without loss of generality, we assume $\bar{\vara} = 0$.
Let $\{e_i\}_1^\infty$ be an
orthonormal basis of eigenvectors of $\mathcal{C}$
with corresponding positive eigenvalues $\{\lambda_i\}_1^\infty$.
By $\pi_n(s) = \proj{s}{j}$ we denote the orthogonal projection
onto the span of the first $n$ eigenvectors.
Observe that
$\ip{\varb}{s}\ip{\K s}{s} = \lim_{n \to \infty} \ip{\varb}{\pi_n(s)}\ip{\K \pi_n(s)}{\pi_n(s)}$,
and that $| \ip{\varb}{\pi_n(s)}\ip{\K \pi_n(s)}{\pi_n(s)}| \leq \norm{\K}\norm{\varb} \norm{s}^3$.
Since $\int_\hilb \norm{s}^3 \, \mu(ds) < \infty$, we can apply
the Lebesgue Dominated Convergence Theorem to obtain
\[
\begin{aligned}
\int_\hilb &\ip{\varb}{s}\ip{\K s}{s}\, \mu(ds) =
\lim_{n \to \infty} \int_\hilb \ip{\varb}{\pi_n(s)}
\ip{\K \pi_n(s)}{\pi_n(s)} \, \mu(ds)\\
&= 
\lim_{n \to \infty} \sum_{i,j,k=1}^n %
\ip{\varb}{e_i} \ip{\K e_j}{e_k}
\int_{\hilb} \ip{s}{e_i}\ip{s}{e_j}\ip{s}{e_k} \, \mu(ds) = 0,
\end{aligned}
\]
where in the last step we used
that the random $n$-vector $\vec{Y}:\hilb \to \R^n$ defined by $\vec{Y}(s) = \big(\ip{s}{e_1}, \ip{s}{e_2}, \ldots, \ip{s}{e_n}\big)$ 
is an $n$-variate Gaussian whose distribution law is
$\mu \circ \vec{Y}^{-1} = \GM{\vec{0}}{\diag(\lambda_1, \ldots, \lambda_n})$.
Note that we also use the
result (see, e.g.,~\cite{Withers85}) that for a mean zero $n$-variate normal random vector $\vec{Y}$, 
$\ave\{Y_{\alpha_1} Y_{\alpha_2} Y_{\alpha_3}\} = 0$ 
for $\alpha_1, \alpha_2, \alpha_3 \in \{1, \ldots, n\}$. 
\end{proof}

\subsubsection{Mean and variance of the quadratic approximation}
Next, we derive expressions for the mean and variance of
$\QoI_\mathup{quad}$
in the infinite-dimensional
Hilbert space setting.
\begin{proposition}\label{prp:quadmoments}
Let $\QoI:(\hilb, \borel(\hilb), \mu) \to (\R, \borel(\R))$ be a function that is
twice differentiable at $\iparb \in \hilb$, with gradient $\Grad(\iparb) \in \hilb$ and
Hessian $\H(\iparb) \in \mathcal{L}_\mathup{sym}(\hilb)$.
Let $\QoI_\mathup{quad}:(\hilb, \borel(\hilb), \mu) \to (\R,
\borel(\R))$ be %
as defined in~\eqref{equ:quadapprox}. Then,
\begin{align}
\ave \{\QoI_\mathup{quad}\}&= \QoI(\iparb) + \frac12 \trace\big[\precond{\Hess(\iparb)}\big],\label{eq:EQoI}\\
\var \{\QoI_\mathup{quad}\}&= \ip{\Grad(\iparb)}{\C[\Grad(\iparb)]} +
    \frac12 \trace\big[ (\precond{\Hess(\iparb)})^2\big].\label{eq:VQoI}
\end{align}
\end{proposition}
\begin{proof}
The first statement follows from,
\[
   \begin{aligned}
   \ave\{\QoI_\mathup{quad}\} = \int_\hilb \QoI_\mathup{quad}(\ipar) \mu(d\ipar)
           &= \QoI(\iparb) + \frac12 \int_\hilb \ip{\Hess(\iparb)(\ipar - \iparb)}{\ipar-\iparb} \mu(d\ipar)\\
           &= \QoI(\iparb) + \frac12 \trace\big[\Hess(\iparb)\C\big]
           = \QoI(\iparb) + \frac12 \trace\big[\precond{\Hess(\iparb)}\big].
   \end{aligned}
\]

To derive the expression for the variance, 
first note that the variance of $\QoI_\mathup{quad}(\ipar)$ equals the variance of 
$\QoI_\mathup{quad}(\ipar) - \QoI(\iparb)$. Thus,
\begin{equation}\label{equ:var}
    \var\{\QoI_\mathup{quad}\} = \ave\{ (\QoI_\mathup{quad}(\ipar) - \QoI(\iparb))^2 \} - \ave\{\QoI_\mathup{quad}(\ipar) - \QoI(\iparb)\}^2. 
\end{equation}
The first term on the right hand side is given by  
\begin{equation*}
\begin{aligned}
\ave\{ &(\QoI_\mathup{quad}(\ipar) - \QoI(\iparb))^2 \}
=
\ave\Big\{ \big(\ip{\Grad(\iparb)}{\ipar - \iparb)} + \frac12 \ip{\Hess(\iparb)(\ipar - \iparb)}{\ipar-\iparb}\big)^2\Big\}\\
& =\ave\big\{ \ip{\Grad(\iparb)}{\ipar - \iparb)}^2\big\} + \frac14 \ave\big\{\ip{\Hess(\iparb)(\ipar - \iparb)}{\ipar-\iparb}^2\big\}\\
&\quad + \ave\big\{ \ip{\Grad(\iparb)}{\ipar - \iparb)} \ip{\Hess(\iparb)(\ipar - \iparb)}{\ipar-\iparb}\big\} \\ 
& = 
\ip{\Grad(\iparb)}{\C[\Grad(\iparb)]} + \frac14
\trace\big[\precond{\Hess(\iparb)}\big]^2 +
    \frac12 \trace\big[ (\precond{\Hess(\iparb)})^2 \big].
\end{aligned}
\end{equation*}
This, along with
$\ave\{\QoI_\mathup{quad}(\ipar) - \QoI(\iparb)\} = \frac12 \trace(\precond\Hess)$ 
and~\eqref{equ:var} finishes the proof.
\end{proof}

Note that the expressions for the mean and variance of
the linear approximation $\QoI_\mathup{lin}$ defined
in~\eqref{equ:linearapprox} consist of only the first terms
in~\eqref{eq:EQoI} and \eqref{eq:VQoI}, respectively.
We also point out an intuitive interpretation of the
\emph{covariance-preconditioned} Hessian, $\precond{\Hess}$, in the expressions
for the mean and variance in~\eqref{eq:EQoI} and \eqref{eq:VQoI}. As the
Hessian $\Hess$ only appears preconditioned by the covariance $\C$, the
second-order contributions to the expectation and the variance are
large only if the dominating eigenvector directions of $\C$ and
$\Hess$ are ``similar''.
More precisely, the eigenvectors of $\C$ that correspond to large eigenvalues
(i.e., directions of large uncertainty) only have a significant influence on
the mean and variance if these directions are also important for the Hessian of
the prediction. Conversely, important directions for the prediction Hessian
only result in significant contributions to the second-order approximation of
mean and variance if the uncertainty in these directions is
significant.

We point out that while the Gaussian assumption on the distribution law of
$\ipar$ is required for derivation of the expression for $\var
\{\QoI_\mathup{quad}\}$ in Proposition~\ref{prp:quadmoments}, the expression
for $\ave \{\QoI_\mathup{quad}\}$ can be derived without this assumption. 
Namely, it holds if the law
of $\ipar$ is any Borel probability measure on $\hilb$ with bounded first and second moments. 
These assumptions on the law of $\ipar$ are also sufficient for  
deriving the expressions for the mean and variance of
$\QoI_\mathup{lin}$,
\[
   \ave\{ \QoI_\mathup{lin} \} = \QoI(\iparb), \quad \var\{  \QoI_\mathup{lin} \} = \ip{\Grad(\iparb)}{\C[\Grad(\iparb)]}.
\]
Here, the expression for the mean is immediate and the one for variance follows
from~\eqref{equ:Cvv}.

\subsubsection{Expected value of the truncation error}
\label{sec:trucation_error}
Next, we discuss the error due to replacing $\QoI$ by
the quadratic approximation $\QoI_\mathup{quad}$. Assuming sufficient
smoothness and boundedness of the derivatives of $\QoI$, we study the
expected value of the truncation error as the uncertainty in the parameter $\ipar$ decreases,
i.e., we consider $\GM{\iparb}{\eps \C}$, where $\eps > 0$
approaches zero.  To gain intuition,  consider the case when the
covariance operator is such that parameter draws $\ipar$ have a high
probability of being close to the mean $\iparb$, where the quadratic
approximation is more accurate. In this case, we can expect the truncation
error to be small.
To derive a quantitative estimate, we assume that $\QoI$ is three
times continuously differentiable such that the remainder term
in the quadratic expansion \eqref{eq:Frechet} has the form:
\begin{equation}\label{eq:remainder}
R(\ipar;\iparb) := \QoI(\ipar) - \QoI_\mathup{quad}(\ipar) = \frac{1}{3!} \QoI^{(3)}(\xi)[\ipar-\iparb,\ipar-\iparb,\ipar-\iparb].
\end{equation}
Here, $\QoI^{(3)}$ denotes the third derivative with respect to
$\ipar$ at $\xi$, which is an element of the line segment
between $\iparb$ and $\ipar$. Note that this and the following
considerations are pointwise in the control variable $\ctrl$.
We are interested in the
expectation value of the remainder term, which we will relate to
$\trace(\C)$, the
average variance of $\ipar$.

Assuming $\QoI^{(3)}$ is a uniformly bounded trilinear map, i.e.,
for all $\xi \in \hilb$,
$ %
  | \QoI^{(3)}(\xi)[u, v, w]| \leq K \norm{u}\norm{v}\norm{w} \text{ for
    all } u, v, w\in \hilb,
$ %
we obtain
\begin{multline}\label{equ:boundCS}
\int_\hilb |R(\ipar; \iparb)|\, \mu(d\ipar) \leq K \int_\hilb \norm{\ipar-\iparb}^3\,\mu(d\ipar)
= K \int_\hilb \norm{\ipar-\iparb}\norm{\ipar-\iparb}^2\,\mu(d\ipar)\\
\leq K \left(\int_\hilb \norm{\ipar-\iparb}^2\, \mu(d\ipar)\right)^{1/2} \left(\int_\hilb\norm{\ipar-\iparb}^4\,\mu(d\ipar)\right)^{1/2}.
\end{multline}
Next, recall that $\int_\hilb \norm{\ipar-\iparb}^2\, \mu(d\ipar) = \trace(\C)$. Moreover, using~\eqref{equ:quadform_2nd_moment} with $\mathcal{K} = I$,
we have 
\begin{equation}\label{equ:bound_forth_moment}
\int_\hilb \norm{\ipar - \bar{\ipar}}^4 \, \mu(d\ipar) = 2 \trace(\C^2) + \trace(\C)^2  \leq  3\trace(\C)^{2},
\end{equation}
where we have also used $\trace(\C^2) \leq  \trace(\C)^2$.
Therefore, using~\eqref{equ:boundCS} and~\eqref{equ:bound_forth_moment}, we have 
\[
\int_\hilb |R(\ipar; \iparb)|\, \mu(d\ipar) 
\leq \sqrt{3}K \trace(\C)^{3/2}.
\]
Now, if we consider a family of laws $\mu_\eps = \GM{\iparb}{\eps \C}$, $\eps > 0$, for $\ipar$, then, the 
expected value of the remainder \eqref{eq:remainder}
is $\mathcal{O}(\eps^{3/2})$, as $\eps \to 0$.
This should be contrasted with the expected value of the remainder term for the linear expansion which can 
be shown to be $\mathcal{O}(\eps)$. Note that if $\QoI$ is cubic, the
third derivative $\QoI^{(3)}$ is constant and the expectation over the
remainder term vanishes since $\ipar$ follows a
Gaussian distribution that is symmetric with respect to its mean.

The above argument regarding the expected value of the remainder in a Taylor expansion can 
be generalized for higher order expansions; see appendix~\ref{apdx:taylor_series}. 

\subsection{The OUU objective function}\label{sec:ouu_objective}
We can now give the explicit form for the objective function for the 
risk-averse OUU problem
\eqref{equ:risk-averse-general}, in which we use $\QoI_\mathup{quad}$
rather than $\QoI$, and thus \eqref{eq:EQoI} and \eqref{eq:VQoI}:
\begin{equation}\label{equ:ouu-prob-general}
\begin{split}
\QoI(z, \iparb) &+ \frac12 \trace\big[\precond{\Hess(z, \iparb)}\big]
   \\&+ \frac\beta 2 \Big\{\ip{\Grad(z, \iparb)}{\C[\Grad(z, \iparb)]} +
    \frac12 \trace\Big[ (\precond{\Hess(z, \iparb)})^2  \Big]\Big\} + \frac{\gamma}{2}\norm{\ctrl}^2.
\end{split}
\end{equation}
Note that we have also added the control cost
$\frac{\gamma}{2}\norm{\ctrl}^2$ in \eqref{equ:ouu-prob-general}.  The
numerical computation of operator traces appearing in the expressions
for the mean and variance of the quadratic approximation is expensive
and can be prohibitive for inverse problems governed by PDEs.  Hence,
we employ approximations obtained by randomized trace
estimators~\cite{AvronToledo11,Roosta-KhorasaniAscher14}, which require
only the application of the operator to (random) vectors and 
provide reasonably accurate trace estimates using a small
number $\Ntr$ of trace estimator vectors (see also~\cite[Appendix A]{AlexanderianPetraStadlerEtAl16}
for a result on an infinite-dimensional Gaussian trace estimator). 
Trace estimation for
$\trace(\precond{\Hess})$ and $\trace\big[(\precond{\Hess})^2\big]$
  amounts to
\begin{equation}\label{equ:trace_est}
\begin{aligned}
    \trace(\precond{\Hess}) &\approx \frac{1}{\Ntr} \sum_{j = 1}^\Ntr \ip{\zeta_j}{\Hess\zeta_j},\\
    \trace\Big[(\precond{\Hess})^2\Big] &\approx \frac{1}{\Ntr} \sum_{j = 1}^\Ntr 
    \ip{\Hess\zeta_j}{\C[\Hess\zeta_j]},
\end{aligned}
\end{equation}
where $\zeta_j$, $j \in \{1, \ldots, \Ntr\}$ are draws from the measure $\nu = \GM{0}{\C}$.
This form of the trace estimators is justified by the identities
\[
\int_\hilb \ip{\Hess \zeta}{\zeta} \, \nu(d\zeta) =
\trace(\precond{\Hess}), \quad
\int_\hilb \ip{\Hess \zeta}{\C [ \Hess \zeta]} \, \nu(d\zeta) = \trace\big[(\precond{\Hess})^2\big].
\]
Replacing the operator traces in the OUU objective
function~\eqref{equ:ouu-prob-general} using trace estimators results
in the OUU objective function
\begin{subequations}\label{equ:ouu-objective-randomized}
\begin{equation}\label{equ:ouu-objective-randomized1}
\begin{aligned}
\J(z) := \QoI(z, \iparb) 
&+ \frac1{2\Ntr} \sum_{j = 1}^\Ntr \ip{\zeta_j}{\psi_j}\\
&+ \frac\beta2 \Big\{\ip{\Grad(z, \iparb)}{\C[\Grad(z, \iparb)]} 
+ \frac1{2\Ntr} \sum_{j = 1}^\Ntr\ip{\psi_j}{\C\psi_j} 
    \Big\} + \frac{\gamma}{2} \norm{z}^2,
\end{aligned}
\end{equation}
where for $j \in \{1, \ldots, \Ntr\}$,
\begin{equation}
\psi_j = \Hess(z, \iparb) \zeta_j, \quad \zeta_j \sim \GM{0}{\C}.
\end{equation}
\end{subequations}

As an alternative to the randomized estimator in~\eqref{equ:trace_est}, 
we can use 
\begin{equation*}
    \trace(\precond{\Hess}) \approx \sum_{j = 1}^\Ntr \ip{w_j}{\Hess w_j},\quad
    \trace\big[(\precond{\Hess})^2\big] \approx \sum_{j = 1}^\Ntr
    \ip{\Hess w_j}{\C[\Hess w_j]},
\end{equation*}
with $w_j = \C^{1/2} v_j$, where $\{ v_j \}_{j = 0}^\infty$ is an orthonormal 
basis of $\hilb$, and $\Ntr$ is an appropriate truncation level. One
possibility is to choose 
$v_j$, $j \in \{1, \ldots, \Ntr\}$ as the  
dominant eigenvectors of $\C^{1/2}{\Hess}(\iparb, z_0)\C^{1/2}$, where
$z_0$ is a nominal control variable. 
Since in many applications the operator $\precond{\Hess}$ has a rapidly
decaying spectrum, $\Ntr$ can be chosen small.  While such an
estimator is tailored
to the Hessian evaluated at $z_0$, we have observed it to
perform well for values of the control variable in a neighborhood of $z_0$. We
will demonstrate the utility of this approach in our computational results.

Note that the existence of minimizers 
for $\J(z)$ (as well as \eqref{equ:ouu-prob-general})
depends on the control space and on properties of $\Grad$ and $\Hess$,
and must be argued on a case-by-case basis.

\renewcommand{\S}{\mathcal{S}}
\subsection{Extensions}
\label{sec:extensions}
In some applications, the control objective $\QoI(\ctrl, \ipar)$ might
be defined only for $\ipar$ in 
a Banach subspace $\mathscr{X} \subset \hilb$ with $\mu(\mathscr{X})=1$. 
Let $(\mathscr{X}, \norm{\cdot}_\mathscr{X})$ be such a subspace, and 
recall that since the Cameron--Martin space $\CM$ is compactly embedded in 
all subspaces of $\hilb$ that have full measure~\cite{Stuart10}, $\CM$ is compactly embedded in $\mathscr{X}$.
Moreover, when using quadratic approximations, we need derivatives at $\iparb \in \CM$. It is thus 
reasonable to require existence of derivatives only in the Cameron--Martin space.
In such cases, one might be tempted to consider the restriction of $\QoI$ 
to the Cameron--Martin space $\CM$ and use 
\begin{equation}\label{equ:taylor_on_cm}
    \QoI( \ipar) \approx \QoI(\iparb) + \ip{\Grad( \iparb)}{\ipar-\iparb}_{\CM\!,\CM^*}
          + \frac12 \ip{\Hess(\iparb)(\ipar-\iparb)}{\ipar-\iparb}_{\CM\!, \CM^*}, \quad \ipar \in \CM. 
\end{equation}
Here $\ip{\cdot}{\cdot}_{\CM,\CM^*}$ denotes the duality pairing between $\CM$ and its dual $\CM^*$, 
and $\Grad(\iparb) \in \CM^*$ and $\Hess(\iparb) \in \mathcal{L}(\CM, \CM^*)$ are the gradient 
and Hessian of $\QoI(\ipar)$ at $\ipar = \iparb$, respectively. Note that we have suppressed the dependence
of $\QoI$ on $\ctrl$. 

The definition~\eqref{equ:taylor_on_cm}, however, is not meaningful from a measure-theoretic point of view as $\CM$ has measure zero.
A possible remedy is to define a bounded and self-adjoint linear operator 
$\S^\delta:\hilb\to\CM$ and to consider the composition
$\QoI^\delta(\ipar) := \QoI(\S^\delta \ipar)$. One possibility is to
choose $\S^\delta:=(I+\delta\C^{-1/2})^{-1}$, in which case the 
smoothing is controlled by $\delta>0$.
The gradient and Hessian of $\QoI^\delta(\cdot)$ are now given 
by
\[
    \QoI^\delta_\ipar(\iparb) = \S^\delta\Grad(\S^\delta \iparb) \in \hilb, \quad 
    \QoI^\delta_{\ipar\ipar}(\ipar) = \S^\delta \Hess(S^\delta\iparb) \S^\delta
    \in \mathcal{L}(\hilb).
\]
This way, one might consider the local quadratic approximation
\[
\QoI^\delta_\text{quad}(\ipar) = \QoI^\delta(\iparb) +
\langle{\QoI^\delta_{\ipar}(\iparb)},{\ipar - \iparb}\rangle 
+ \frac12 \langle{\QoI^\delta_{\ipar\ipar}(\iparb)(\ipar - \iparb)},{\ipar-\iparb}\rangle, 
\quad \ipar \in \hilb.
\]
This construction allows to consider control objectives that are
defined %
only on $\CM$ to be extended to $\hilb$ via the mapping $\QoI^\delta$. Then, the Hilbert space formulation 
of the risk-averse OUU with quadratic approximations, developed in earlier sections, can be applied.

Another case where the Hilbert space theory needs extension is
when $\QoI$ is defined on a Banach
subspace $\mathscr{X}$ of full measure, and thus its gradient and
Hessian belong to 
$\mathscr{X}^*$ and $\mathcal{L}(\mathscr{X}, \mathscr{X}^*)$, respectively.
For example, in the control problem governed by a linear elliptic PDE
with uncertain coefficient discussed in section~\ref{sec:app},
the space $\mathscr{X} = C(\D)$ plays such a role (it is
known~\cite{Stuart10,DashtiStuart15} that due to our choice of the covariance
operator $\C$, $\mu\big(\mathscr{X}\big) = 1$). In
this case, we show that the expressions for the mean and variance of the
quadratic approximation continue---under appropriate assumptions---to
be well-defined.

Now, the gradient $\Grad(\iparb) \in \mathscr{X}^* \subset \CM^*$, and
thus $\ip{\Grad(\iparb)}{\C[\Grad(\iparb)]}$  
can be interpreted as a duality product, i.e., the
linear action of $\Grad(\iparb) \in \CM^*$ on $\C[\Grad(\iparb)] \in
\CM$. Here, we have used that for the covariance operator 
$\C$ defined above, we have $\C^{1/2}:\CM^*
\to \CM$.

Next, considering the expressions \eqref{eq:EQoI} and \eqref{eq:VQoI}
for mean and variance,
it remains to specify conditions that ensure that the operator $\C^{1/2} \Hess(\iparb) \C^{1/2}$ is
trace-class on $\hilb$. %
\begin{proposition} 
Let $\C$ be the covariance operator as defined in section~\ref{sec:measures}. 
Assume that $\Hess(\iparb) \in \mathcal{L}(\mathscr{X}, \mathscr{X}^*)$ restricted to $\CM$ is a bounded linear operator on 
$(\CM, \norm{\cdot}_\CM)$.
Then, the operator $\C^{1/2} \Hess(\iparb) \C^{1/2}$ is a trace class operator on $\hilb$. 
\end{proposition}
\begin{proof}
It is straightforward to see that $\C^{1/2} \Hess(\iparb) \C^{1/2} \in \mathcal{L}(\hilb)$.
It remains to show that $\C^{1/2} \Hess(\iparb) \C^{1/2}$ is trace-class.
Let $\{ e_j \}_{j = 1}^\infty$ be the complete orthonormal set of eigenvectors of $\C$, with corresponding
(positive) eigenvalues $\{\lambda_j\}_{j = 1}^\infty$. 
We note that for $u, v \in \hilb$, $\ip{u}{v} = \cip{\C^{1/2}u}{\C^{1/2} v}$. Therefore, 
for each $j \geq 1$, we can write 
\begin{multline*}
    \ip{e_j}{\C^{1/2} \Hess(\iparb) \C^{1/2} e_j} 
    = \cip{\C^{1/2} e_j}{\C \Hess(\iparb) \C^{1/2} e_j}
    \leq \norm{\Hess(\iparb)} \|{\C^{1/2}\C  e_j}\|_\CM \|{\C^{1/2} e_j}\|_\CM
    \\
    = \norm{\Hess(\iparb)}  \norm{\C  e_j}\norm{e_j}
= \norm{\Hess(\iparb)}\lambda_j.
\end{multline*}
Therefore, $\sum_j \ip{e_j}{\C^{1/2} \Hess(\iparb) \C^{1/2} e_j} \leq \norm{\Hess(\iparb)} \sum_j \lambda_j < \infty$.
\end{proof}

\section{Control of a semilinear elliptic PDE with uncertain Neumann boundary data} \label{sec:semilinear}
We first illustrate our approach for
the optimal control of a semilinear elliptic PDE with uncertain
Neumann boundary data and a right hand side control. In this problem,
the nonlinearity in the governing PDE is the sole reason why the
quadratic approximation $\QoI_\mathup{quad}(\cdot)$ of the objective is not
exact. Below, we present and discuss the optimization formulation for this
PDE-constrained OUU problem.

We assume a bounded domain $\D\subset
\mathbb R^2$ with boundary split into disjoint parts $\GN$ and $\GD$,
and we consider the semilinear elliptic equation
\begin{equation}\label{eq:semistate}
  \begin{aligned}
  -\Delta u + cu^3 &= \ctrl &\quad\text{ in }& \D,\\
  u &= 0 &\quad\text{ on }& \GD,\\
  \nabla u\cdot\vec{n} &=\ipar &\quad\text{ on }& \GN.
  \end{aligned}
\end{equation}
Here, $c\ge 0$, $\ctrl\in L_2(\D)$ is the control and the uncertain
parameter is $\ipar\in L^2(\GN)$, distributed according to the law
$\mu = \GM{\iparb}{\C}$ with mean $\iparb$ and covariance operator
$\C$. Due to the monotonicity of the nonlinear term, the state
equation \eqref{eq:semistate} has a unique solution for every Neumann
data $\ipar$ and every $\ctrl\in L_2(\D)$,  \cite{Reyes15}.

We consider a control objective of tracking type as follows:
\begin{equation}\label{eq:semiQoI}
\QoI(\ctrl,\ipar) = \frac 12 \| u - u_d\|^2, %
\end{equation}
where $u_d\in L^2(\D)$ is a given desired state. 
It is straightforward 
to show that the solution $u$ of~\eqref{eq:semistate} satisfies the estimate 
$\norm{u}_{H^1(\D)} \leq K  (\norm{z}_{L^2(\D)} + \norm{m}_{L^2(\GN)})$,
for a constant $K = K(\D)$.
Hence, since $\ipar$ has moments of all orders, it follows that the state variable $u = u(\ctrl,\ipar)$ 
also has moments of all orders for every $\ctrl \in L^2(\D)$. This in particular implies the existence
of the first and second moments of the control objective $\QoI(\ctrl, \ipar)$ for every $\ctrl$. 
Therefore, a mean-variance risk-averse OUU objective function is well-defined. 

To derive the
quadratic approximation of \eqref{eq:semiQoI} with respect to $\ipar$
at the mean $\iparb$,
we compute, for fixed control $\ctrl$, the gradient and Hessian of
$\QoI$ with respect to $\ipar$.
The gradient of $\QoI$ with respect to $\ipar$ is
$\QoI_\ipar(\iparb)=-\bar p_{|\GN}$, where $u$ satisfies
\eqref{eq:semistate} with $\ipar=\iparb$ (the corresponding state is
denoted by $\bar u$), and $\bar p$ satisfies the adjoint equation
\begin{equation}\label{eq:adj}
  \begin{aligned}
  -\Delta \bar p + 3c\bar u^2\bar p &= -(\bar u-u_d) &\quad\text{ in }& \D,\\
  \bar p &= 0 &\quad\text{ on }& \GD,\\
  \nabla \bar p\cdot\vec{n} &= 0 &\quad\text{ on }& \GN.
  \end{aligned}
\end{equation}
The second derivative at $\iparb$ evaluated in a direction $\hat
\ipar$ is given by $\QoI_{\ipar\ipar}(\iparb)(\hat \ipar)=-\hat
p_{|\GN}$, where $\hat p$ solves the incremental adjoint equation:
\begin{equation}\label{eq:incadj}
  \begin{aligned}
  -\Delta \hat p + 3c\bar u^2\hat p  &= -(6c\bar u \bar p +1) \hat u &\quad\text{ in }& \D,\\
  \hat p &= 0 &\quad\text{ on }& \GD,\\
  \nabla \hat p\cdot\vec{n} &=0 &\quad\text{ on }& \GN,
  \end{aligned}
\end{equation}
and $\hat u$ the incremental state equation:
\begin{equation}\label{eq:incstate}
  \begin{aligned}
  -\Delta \hat u + 3c\bar u^2\hat u &= 0 &\quad\text{ in }& \D,\\
  \hat u &= 0 &\quad\text{ on }& \GD,\\
  \nabla \hat u\cdot\vec{n} &= \hat\ipar &\quad\text{ on }& \GN.
  \end{aligned}
\end{equation}
We now show that $\Hess(\iparb):\ipar\mapsto -\hat p_{|\GN}$ is
bounded as mapping from $L_2(\GN)\to L_2(\GN)$. From
\eqref{eq:incstate} it follows that $\|\hat u\|_{H^1(\D)}\le
\|\hat m\|$. To estimate the $H^{-1}(\D)$-norm of the right hand
side in \eqref{eq:incadj}, we consider an arbitrary $v\in H^1(\D)$.
Using H\"older's inequality and the continuous embedding
of $H^1(\D)$ in $L^4(\D)$, we obtain
\begin{align*}
\int_\D (6c\bar u\bar p+1) \hat u v \, dx &\le \|6c\bar u\bar
p+1\| \|\hat u v \| \le \left(c_1+c_2\|\bar u\|_{L^4(\D)}\|\bar
p\|_{L^4(\D)}\right) \|\hat
u\|_{L^4(\D)}\|v\|_{L^4(\D)}\\
&\le c_3\|\hat u \|_{H^1(\D)}
\|v \|_{H^1(\D)},
\end{align*}
where the constants $c_1, c_2, c_3$ do not depend on $\hat u$ or $v$.
This shows that the $H^{-1}(\D)$-norm of the right hand side in
\eqref{eq:incstate} is bounded by $\|\hat u\|_{H^1(\D)}$. Hence,
with constants $c_4,c_5,c_6$:
\begin{equation*}
\|\hat p_{|\GN} \|_{L^2(\GN)}\le c_4\|\hat p\|_{H^1(\D)}\le c_5\|\hat
u\|_{H^1(\D)}\le c_6\|\hat m\|_{L^2(\GN)},
\end{equation*}
which proves the boundedness of $\Hess$.
Thus, the OUU objective function~\eqref{equ:ouu-prob-general}, 
specialized to the present example, 
with the gradient and Hessian operators $\Grad(\iparb)$ and $\Hess(\iparb)$ 
defined above is well-defined and conforms to the theory 
outlined in sections~\ref{sec:moments_of_quad_approx}--\ref{sec:ouu_objective}.

\section{Control of an elliptic PDE with uncertain coefficient}
\label{sec:app}
Motivated by problems involving the optimal control of flows in porous
media, we consider the optimal control of a linear elliptic PDE with
uncertain coefficient field.  We discuss this application
numerically in section~\ref{sec:numerics}, where we consider control
of fluid injection into the subsurface at injection wells. In this section,
we describe the control objective and the PDE-constrained objective
function for the risk-averse OUU problem, and derive the adjoint-based
expressions for the gradient of the OUU objective function.
This is an example for a problem in which the
parameter-to-objective map is defined only on a Banach 
subspace of $\hilb$ that has full measure.  We thus use the expressions for
the mean and variance of the quadratic approximation, developed in a
Hilbert space setting in section~\ref{sec:analytic}, formally, to
define the objective function for the risk-averse optimal control
problem.

We begin by describing the state (forward) equation. 
On an open bounded and sufficiently smooth domain $\D\subset\R^n$,
$n\in\{1,2,3\}$ with boundary $\partial\D$, we consider the following
elliptic partial differential equation:
  \begin{equation}\label{equ:poi}
    \begin{aligned}
      -\nabla \cdot (\Exp{m} \nabla u) &= b + F\ctrl  &&\text{ in }\D, \\
                                        u  &= g  &&\text{ on } \GD, \\
           \Exp{m} \nabla u \cdot \vec{n} &= 0  &&\text{ on } \GN.
    \end{aligned}
  \end{equation}
Here, the boundary is split into disjoint parts $\GD$ and $\GN$ on
which we impose Dirichlet and Neumann boundary conditions, respectively. The
Dirichlet data is $g\in H^{1/2}(\GD)$, $\vec{n}$ denotes the
unit-length outward normal for the boundary $\partial\D$, and for simplicity we
have considered homogeneous Neumann conditions.
We assume that the right hand side is specified as 
$Fz + b$, where $F:L^2(\D) \to L^2(\D)$ is a bounded linear transformation
and $b \in L^2(\D)$ is a distributed source per unit volume.

We consider the weak form of \eqref{equ:poi}, i.e., we seek solutions
$u\in \Vg := \{ v \in H^1(\D) : \restr{v}{\GD} = g\}$ that satisfy
\begin{equation}\label{equ:poiweak}
\ip{\Exp{m} \nabla u}{\nabla v} - \ip{b + F\ctrl}{v} = 0
\quad\text{ for all } v\in 
\V :=  \{ v \in H^1(\D) : \restr{v}{\GD} = 0\}.
\end{equation}
We consider the case where the log-permeability, $\ipar$, is uncertain and is
modeled as a spatially distributed random field; see also \cite{DashtiStuart15,BonizzoniNobile14}. The realizations of $\ipar$
belong to the Hilbert space $\hilb = L^2(\D)$, and we assume $\ipar$ is
distributed according to a Gaussian with covariance operator $\C$ and mean
$\iparb\in \CM \subset \hilb$, where $\CM$ is the
Cameron--Martin space associated with the Gaussian measure $\mu$.  
With our choice of the covariance operator, $\mu(\mathscr{X}) = 1$ with $\mathscr{X}
= C(\bar{\D}) \subset \hilb$.

We consider the following tracking-type control objective $\QoI(\ctrl,\ipar)$:
\begin{equation}\label{equ:controlobj}
   \QoI(\ctrl,\ipar) =\frac12 \euclidnorm{ \Q u(\ctrl,\ipar) - \obsq}^2,
\end{equation}
where $u = u(\ctrl, \ipar)$ solves the weak form of the state
equation~\eqref{equ:poiweak}, $\Q:L^2(\D)\to\R^q$ is a bounded linear operator, and $\obsq\in \R^q$ is given.
Notice that this control objective is defined for every $\ipar \in
\mathscr{X}$, i.e., almost surely. It is also possible to prove
boundedness of moments of $\norm{ u }_\mathscr{X}$, which, in particular,
ensures that a mean-variance risk-averse OUU objective based on \eqref{equ:controlobj} is well-defined. Compared to the problem discussed in
section \ref{sec:semilinear}, here the proof of
boundedness of moments of $\norm{ 
u }_\mathscr{X}$ is more involved (see \cite[Example
2.15]{DashtiStuart15}) and requires the use of Fernique's
theorem~\cite{Fernique70}.

For fixed parameter $\ipar$ and control $\ctrl$, the gradient of the
parameter-to-objective map can be computed with a standard
variational calculus approach (see, e.g.,~\cite{BorziSchulz12}). 
In particular, at $\ipar = \iparb$, the gradient $\Grad(\iparb) \in \mathscr{X}^*$ is given by 
\newcommand{\incu}{\upupsilon}
\newcommand{\incp}{\uprho}
\begin{equation}\label{equ:grad_inner}
\Grad(\iparb) = \!{\Exp{\iparb}\nabla u}\cdot{\nabla p},
\end{equation}
where $u$ is the solution to the state equation, and $p$ solves
the adjoint equation, i.e., $p \in \V$ 
and satisfies
\begin{equation}\label{equ:adj}
  \ip{\Exp{\iparb} \nabla p}{\nabla \ut{p}} = -\ip{\Q^*(\Q u - \obsq)}{\ut{p}}, 
  \quad \text{for all } \ut{p} \in \V.
\end{equation}
Similarly, the action of the Hessian, $\Hess(\iparb)$ in
a direction $\zeta \in \hilb$ can be expressed as
\begin{equation}\label{equ:hess_inner}
   \Hess(\iparb) \zeta \\
     = \Exp{\iparb} (\zeta \nabla u \cdot \nabla p + \nabla \incu \cdot \nabla p + \nabla u \cdot \nabla\incp),
\end{equation}
where $u$ and $p$ are the state and adjoint variables computed with a
given control $\ctrl$ 
at $\ipar = \iparb$. The variables $\incu$ and $\incp$, 
which we refer to as the \emph{incremental} state and adjoint variables, are
obtained by solving the following incremental state and adjoint equations:
\begin{subequations}\label{equ:inc}
\begin{align}
    \ip{\Exp{\iparb} \nabla \incu}{\nabla \ut{\incu}} &= -\ip{\zeta\Exp{\iparb} \nabla u}{\nabla \ut{\incu}} 
    \quad \text{for all } \ut{\incu} \in \V, \\
    \ip{\Exp{\iparb} \nabla \incp}{\nabla \ut{\incp}} &= -\ip{\Q^*\Q \incu}{\ut\incp} - \ip{\zeta\Exp{\iparb} \nabla p}{\nabla \ut{\incp}} 
    \quad \text{for all } \ut{\incp} \in \V.  
\end{align}
\end{subequations}
Note that \eqref{equ:grad_inner} and
\eqref{equ:hess_inner} for $\Grad$ and $\Hess$ require that the state
and adjoint equations \eqref{equ:poi} and \eqref{equ:adj}, as well as
their incremental variants \eqref{equ:inc}, are satisfied.  Thus, in
the formulation of the OUU objective function
\eqref{equ:ouu-objective-randomized1},
these equations must be enforced as constraints.

\newcommand{\nuexpd}{\C^{1/2}\left[ \Exp{\iparb} (\zeta \nabla u \cdot \nabla p + \nabla \incu \cdot \nabla p + \nabla u \cdot \nabla\incp)\right]}
\newcommand{\nuexpdj}{\C^{1/2}\left[ \Exp{\iparb} (\zeta_j \nabla u \cdot \nabla p + \nabla \incu_j \cdot \nabla p + \nabla u \cdot \nabla\incp_j)\right]}
\newcommand{\gexpd}{ \Exp{\iparb}\nabla u \cdot\nabla p}
\newcommand{\psiexpdn} {\Exp{\iparb} (\zeta \nabla u \cdot \nabla p + \nabla \incu \cdot \nabla p + \nabla u \cdot \nabla\incp)}
\newcommand{\psiexpdnj} {\Exp{\iparb} (\zeta_j \nabla u \cdot \nabla p + \nabla \incu_j \cdot \nabla p + \nabla u \cdot \nabla\incp_j)}
\newcommand{\psiexpdnk} {\Exp{\iparb} (\zeta_k \nabla u \cdot \nabla p + \nabla \incu_k \cdot \nabla p + \nabla u \cdot \nabla\incp_k)}
\newcommand{\expduj}{\Exp{\iparb} (\zeta_j \nabla \ut{u} \cdot \nabla p + \nabla \ut{u} \cdot \nabla\incp_j)}
\newcommand{\expdpj}{\Exp{\iparb} (\zeta_j \nabla u \cdot \nabla \ut{p} + \nabla \incu_j \cdot \nabla \ut{p})}

\subsection{The OUU problem for \eqref{equ:poi}}
\label{subsec:ouupb}

Here we summarize the formulation of the OUU problem, which uses a quadratic
approximation of $\QoI$ defined in \eqref{equ:controlobj}:
\begin{subequations}\label{equ:OUUpoi}
  \begin{equation}\label{equ:ouu_cost_quad}
\begin{aligned}
\min_{\ctrl \in Z} \J(\ctrl) &:= \frac12 \euclidnorm{ \Q u - \obsq}^2
+ \frac1{2\Ntr} \sum_{j = 1}^\Ntr \ip{\zeta_j}{\psi_j}\\
&+ \frac\beta2\Big\{\ip{\Grad(\iparb)}{\C[\Grad(\iparb)]}
+ \frac1{2\Ntr} \sum_{j = 1}^\Ntr \ip{\psi_j}{\C \psi_j}
\Big\}
+ \frac\gamma2 \norm{\ctrl}^2,
\end{aligned}
  \end{equation}
where for $j = 1, \ldots, \Ntr$
\begin{align}
\Grad(\iparb) &= \!{\Exp{\iparb}\nabla u}\cdot{\nabla p},\label{equ:risk-averse-ouu-g} \\
\quad \psi_j &= \Exp{\iparb} (\zeta_j \nabla u \cdot \nabla p + \nabla \incu_j \cdot \nabla p + \nabla u \cdot \nabla\incp_j), \label{equ:risk-averse-ouu-Hzeta}
\end{align}
and the variables $(u, p, \{\incu_j\}, \{\incp_j\})\in \V^2 \times (\V^\Ntr)^2$, 
which can be considered the state variables of the OUU problem, solve
\begin{align}
    \ip{\Exp{\iparb} \nabla u}{\nabla \ut{u}} &= \ip{b + F\ctrl}{\ut{u}} &\forall \ut{u} \in \V,  \label{equ:risk-averse-ouu-fwd}\\
    \ip{\Exp{\iparb} \nabla p}{\nabla \ut{p}} &= -\ip{\Q^*(\Q u - \obsq)}{\ut{p}} &\forall \ut{p} \in \V,  \label{equ:risk-averse-ouu-adj}\\
    \ip{\Exp{\iparb} \nabla \incu_j}{\ut{\incu}} &= -\ip{\zeta_j \Exp{\iparb} \nabla u}{\nabla \ut{\incu}} &\forall{\ut\incu \in \V},
    \label{equ:risk-averse-ouu-inc-fwd}\\
    \ip{\Exp{\iparb} \nabla \incp_j}{\ut\incp} &= -\ip{Q^*Q \incu_j}{\ut\incp} - \ip{\zeta_j \Exp{\iparb} \nabla p}{\ut\incp}
    &\forall \incp \in \V. 
    \label{equ:risk-averse-ouu-inc-adj}
\end{align}

\end{subequations}

\subsection{Evaluation and gradient computation of the OUU objective function}
\label{subsec:evalgrad}

Solving the PDE-constrained optimization problem~\eqref{equ:OUUpoi}
efficiently requires gradient-based optimization methods. To compute
the gradient of $\J$
with respect to $\ctrl$, we follow a Lagrangian approach, and employ
adjoint variables (i.e., Lagrange multipliers) to enforce the PDE
constraints~\eqref{equ:risk-averse-ouu-fwd}--\eqref{equ:risk-averse-ouu-inc-adj}.
The details of the derivation of the gradient are relegated to
appendix~\ref{appnd:grad}. 
The 
expression for the gradient, in a direction $\ut z$, takes the form:
\begin{equation}\label{equ:ouu-reduced-gradient}
   \G(z)\ut{z} = \gamma \ip{z}{\ut{z}} - \ip{F\ut{z}}{\ad{u}}, \quad
   \ut{z} \in L^2(\D),
\end{equation}
where $\ad{u}$ is obtained by solving the following system of
equations for the OUU
adjoint variables $(\ad{u}, \ad{p}, \{\ad\incu_j\}, \{\ad\incp_j\}) \in \V^2 \times (\V^\Ntr)^2$
\begin{subequations}\label{equ:raouu-adj-gen}
  \begin{align}
    \!\!\ip{\Exp{\iparb} \nabla \ad{\incp}_j}{\nabla \ut{\incp}}
      &\!=\!\ip{b_1^{(j)}}{\ut{\incp}},
    \label{equ:ouu-adj-incp-simp-gen}
    \\
    \!\!\ip{\Exp{\iparb} \nabla \ad{\incu}_j}{\nabla \ut{\incu}}
    + \ip{\Q^*\Q \ad\incp_j} {\ut{\incu}}
    &\!=\!\ip{b_2^{(j)}}{\ut{\incu}},
    \label{equ:ouu-adj-incu-simp-gen}
    \\
    \!\!\!\ip{\Exp{\iparb} \nabla \ad{p}}{\nabla \ut{p}}
    + \sum_{j = 1}^\Ntr \ip{\zeta_j\Exp{\iparb} \nabla \ad{\incp}_j}{\nabla \ut{p}}
    &\!=\!\ip{b_3}{\ut p},
    \label{equ:ouu-adj-p-simp-gen}
    \\
    \!\! \ip{\Exp{\iparb} \nabla \ad{u}}{\nabla \ut{u}}
    + \ip{\Q^*\Q \ad{p}}{\ut{u}}
    + \sum_{j = 1}^\Ntr \ip{\zeta_j\Exp{\iparb} \nabla \ad\incu_j}{\nabla \ut{u}}
    &\!=\! \ip{b_4}{\ut u},
    \label{equ:ouu-adj-u-simp-gen}
  \end{align}
\end{subequations}
for all $(\ut u, \ut p, \ut{\incu}, \ut{\incp}) \in \V^4$,
with $\{b_1^{(j)}\}_{j=1}^\Ntr ,\{ b_2^{(j)}\}_{j=1}^\Ntr, b_3$, and $b_4$ given by

\begin{align*}
    \ip{b_1^{(j)}}{\ut{\incp}} &=
       -\biggl[\frac1{2\Ntr}\ip{\zeta_j}{\Exp{\iparb} \nabla u \cdot \nabla \ut{\incp}}
       + \frac\beta{2\Ntr} \ip{\Exp{\iparb} \nabla u \cdot \nabla \ut{\incp} }{\C\psi_j }\biggr],
    \\
    \ip{b_2^{(j)}}{\ut{\incu}} &=
      -\biggl[\frac1{2\Ntr} \ip{\zeta_j}{\Exp{\iparb} \nabla \ut{\incu} \cdot \nabla p}
      +\frac\beta{2\Ntr}\ip{\Exp{\iparb}\nabla \ut{\incu} \cdot \nabla
        p}{\C\psi_j }\biggr], %
    \\
    \ip{b_3}{\ut p} &=
    -\biggl[\frac1{2\Ntr} \sum_{j = 1}^\Ntr \ip{\zeta_j}{\expdpj}
     + \beta \ip{\Exp{\iparb}\nabla u\cdot \nabla \ut{p}}{\C[ \Grad(\iparb)]}\notag\\
    &+ \frac\beta{2\Ntr} \sum_{j = 1}^\Ntr \ip{\Exp{\iparb} (\zeta_j \nabla u \cdot \nabla \ut{p}
     + \nabla \incu_j \cdot \nabla \ut{p})}{ \C\psi_j}\biggr], %
    \\
    \ip{b_4}{\ut u} &=
    -\biggl[
     \ip{\Q^* (\Q u - \obsq)}{\ut{u}}
     +  \frac1{2\Ntr}  \sum_{j = 1}^\Ntr \ip{\zeta_j}{\expduj}
     +  \beta \ip{\Exp{\iparb}\nabla \ut{u}\cdot \nabla p}{\C[ \Grad(\iparb)]}\notag\\
     &+ \frac\beta{2\Ntr}  \sum_{j = 1}^\Ntr \ip{\Exp{\iparb} (\zeta_j \nabla \ut{u} \cdot \nabla p
      + \nabla \ut{u} \cdot \nabla \incp_j)}{ \C\psi_j}
     \biggr]. %
\end{align*}
Next, we count the number of PDE solves required for the evaluation of
the objective function \eqref{equ:ouu_cost_quad} and of its gradient
\eqref{equ:ouu-reduced-gradient}. Although in the present example, the
application of $\C$ requires PDE solves that are similar to those in
the state and adjoint equations, we do not include them in our
counting since (1) the matrix in $\C$ does not change and thus a
sparse factorization or multigrid solver can be set up upfront, and (2)
we aim at problems where the state (and thus the adjoint) equation is
more complex than in the present case and thus its solution dominates
the application of $\C$.

Hence, the number of forward-like PDE solves (i.e., the state equations, the
adjoint equation, and incremental equations) required for evaluating the
objective function \eqref{equ:ouu_cost_quad} is $2 + 2 \times \Ntr$. The
gradient evaluation requires another $2 + 2 \times \Ntr$ PDE
solves. Thus, $4 + 4\times \Ntr$ PDE solves are necessary for
evaluating the OUU objective
function and its gradient. If a small $\Ntr$ provides an approximation of
the trace that is suitable for computing an optimal control, each iteration of
a gradient-based method requires a moderate, fixed number
of PDE solves.
We
demonstrate this in our numerical experiments, where for a typical OUU
problem with a discretized parameter dimension of about $3{,}000$, $O(10)$ PDE
solves are required for approximating the OUU objective and its gradient.
This modest computational cost should be contrasted with methods that
approximate the OUU objective using sampling or quadrature in the
parameter space.  While these approaches can, asymptotically, solve
the original OUU problem rather than the formulation based on the
quadratic expansion of $\QoI$, their cost in terms of PDE solves can
be much higher. Moreover, if a factorization-based direct solver can
be used for the governing PDE problems, these factorizations can be
reused multiple times in the quadratic approximation approach. This is
the case since the PDE operators arising in each quadratic
approximation correspond to the same parameter. Thus, reusing
factorizations can save significant computation time. Since each
sample in the Monte Carlo approach corresponds to a different
parameter, matrix factorizations must be recomputed for each sample.

The computational cost of sampling-based methods is exacerbated for
OUU problems governed by nonlinear forward PDEs; in such cases, while
our approach requires only one nonlinear PDE solve and
$\mathcal{O}(\Ntr)$ linear(ized) PDE solves (for OUU objective and
gradient evaluation), the required number of nonlinear PDE solves for
sampling approaches scales with the number of samples.

\section{Computational experiments}\label{sec:numerics}

Next, we numerically study our %
OUU approach applied to the control of an elliptic PDE
with an uncertain coefficient as discussed in section~\ref{sec:app}.  

\subsection{Problem setup}
\label{subsec:setup}
We consider a rectangular domain $\D = (0, 2) \times (0, 1)\subset\R^2$ for
\eqref{equ:poi}, impose Dirichlet boundary conditions on the left and right
sides of the domain according to $u(0, y) = 1$ and $u(2, y) = 0$ for $y \in [0,
1]$, and use homogeneous Neumann boundary conditions on the top and bottom
boundaries.  We use a finite element mesh with $3{,}081$ linear rectangular
elements and $3{,}200$ degrees of freedom to discretize the state and
adjoint variables, and the uncertain parameter field.  The discretized
uncertain parameters are the coefficients in the finite element expansion of
the parameter field.

\paragraph{The probability law of the uncertain parameter field}
The probability law of the uncertain parameter, which is the
log-permeability field $\ipar$, is given by a Gaussian measure $\mu =
\GM{\iparb}{\C}$.  The mean $\iparb$ and the locations of
production and control wells are shown in
Figure~\ref{fig:basic_setup} (upper left image).
The covariance operator $\C$ of $\mu$
is as in \eqref{eq:Cconcrete}, with
$\kappa = 2\times10^{-2}$ and $\alpha = 4$. These values are
chosen such that samples from $\mu$ have a desired magnitude and
correlation length. Typical realizations of the random process
$\ipar$ are also shown in
Figure~\ref{fig:basic_setup}.

\begin{figure}[t]\centering
\includegraphics[width=.442\textwidth]{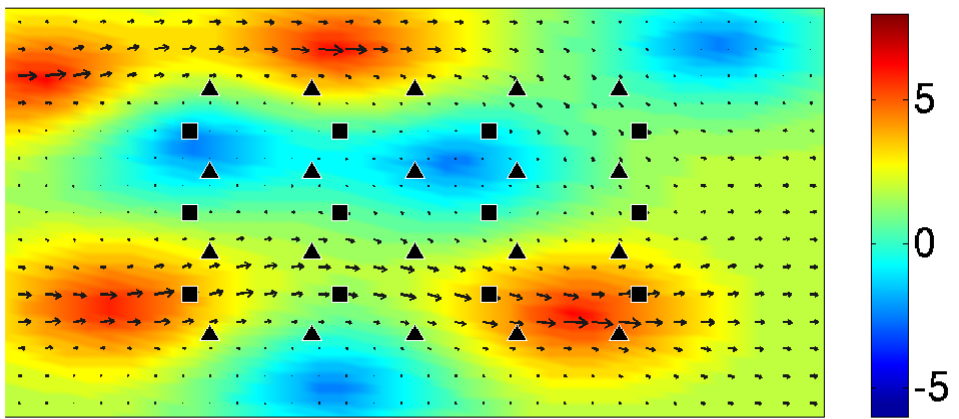}
\includegraphics[height=.21\textwidth,width=0.473\textwidth]{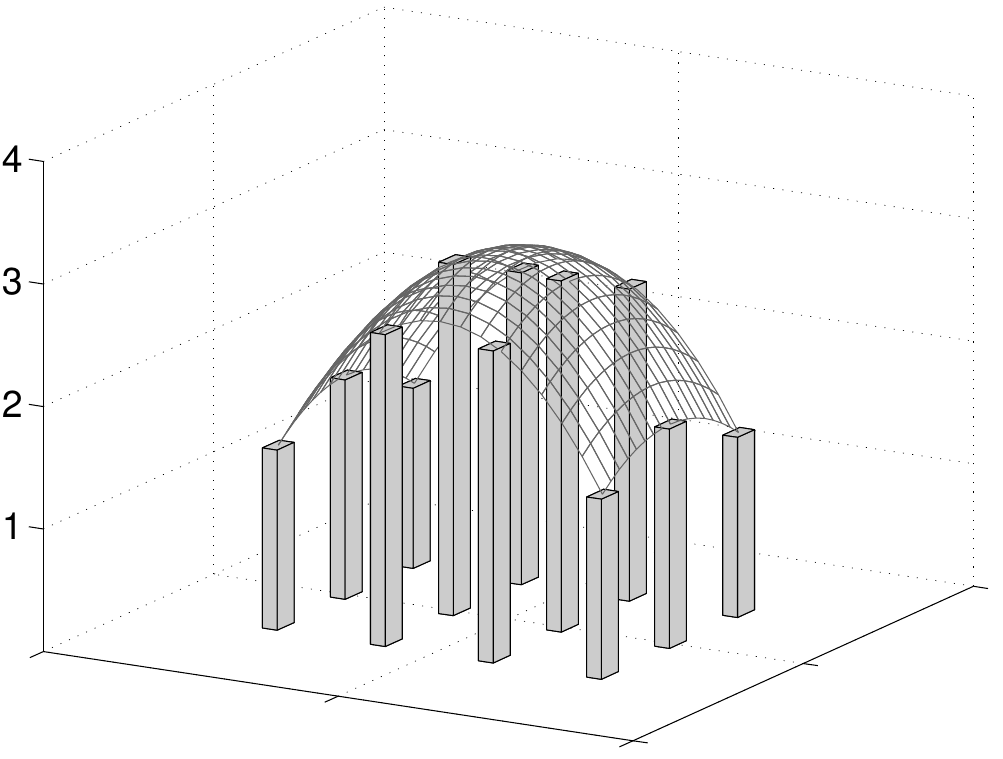}
\\
\vspace{2mm}
\includegraphics[width=.3\textwidth]{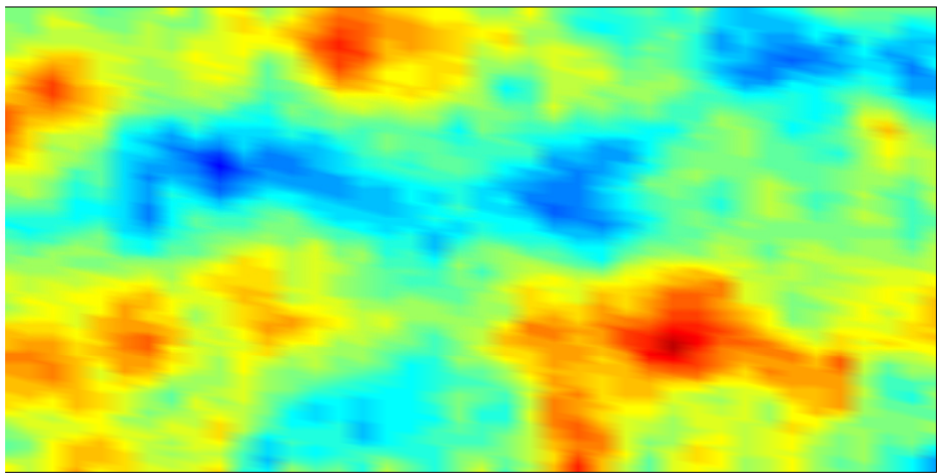}
\includegraphics[width=.3\textwidth]{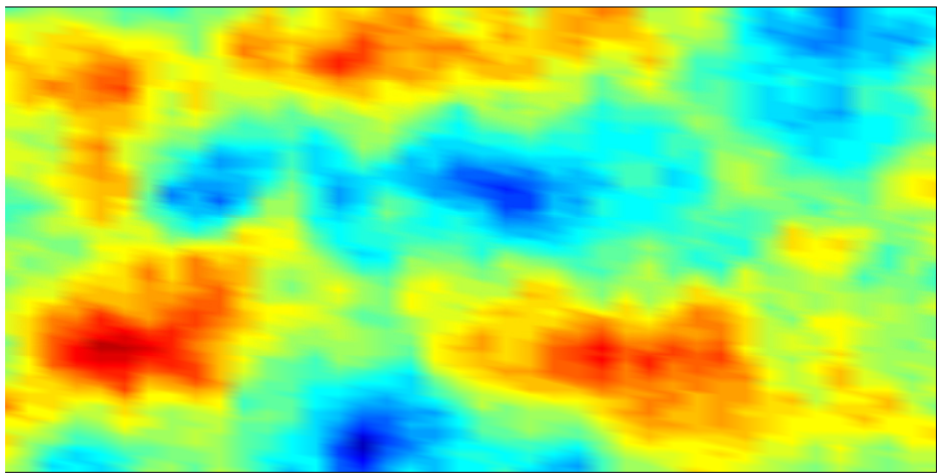}
\includegraphics[width=.3\textwidth]{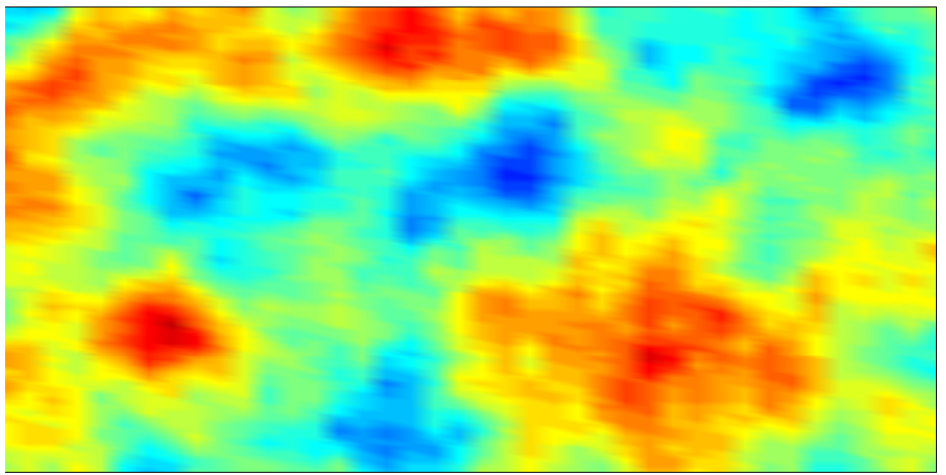}
\caption{Top row: Shown on the left is the mean log permeability field with corresponding
  Darcy flow driven by the lateral boundary conditions. Superimposed is
  the well configuration (triangles indicate the locations of
  injection/control wells, and squares of production wells). Shown on
  the right is the target pressure at the production wells, located at points
  $\{\vec{x}_i\}_{i = 1}^{12}$,  
  computed according to $\bar{q}_i = q(\vec{x}_i)$, where 
  $q(\vec{x}) = 3 - 4(x_1-1)^2-8(x_2-0.5)^2$.
Bottom row: Realizations of the uncertain log-permeability field.
}
\label{fig:basic_setup}
\end{figure}

\paragraph{The control}
For the right hand side in \eqref{equ:poi}, we consider $b \equiv 0$, and
define $F$ as a
weighted sum of finitely many mollified Dirac delta functions, which we
denote by $f_i$, $i = 1, \ldots, \Nc$:
\begin{equation}\label{eq:forcing}
    F\vec{\ctrl} = \sum_{i = 1}^\Nc z_i f_i.
\end{equation}
The weights $z_i$ in \eqref{eq:forcing} are the (finitely many) control
variables, summarized in the control vector $\vec{\ctrl} = (\ctrl_1,
\ldots, \ctrl_{\Nc})^T$.  The right hand side \eqref{eq:forcing} is motivated
by subsurface flow applications in petroleum engineering, where
fluid injection at injection wells is used to control the flow rates at production
wells. Here, $f_i=f_i(x,y)$ represent the locations of injection wells,
and $z_i$ are the injection rates at these wells. We
impose control bounds, i.e., the set of
admissible controls is
\[
   \Zad = \{ \vec{\ctrl} \in \R^\Nc : z_\text{min} \leq \ctrl_i \leq z_\text{max}, 1 \leq i \leq \Nc \},
\]
where $z_\text{min} = 0$ and $z_\text{max} = 16$. 
In the present example, we use $\Nc = 20$ control wells; see 
Figure~\ref{fig:basic_setup} (top left).
The control objective \eqref{equ:controlobj} is the squared $\ell^2$ difference
between the pressure at the production wells and a vector $\obsq$ of
target pressure values, which follows a parabolic profile as
depicted in Figure~\ref{fig:basic_setup} (top right).

\paragraph{Quadratic approximations in the small variance limit}
Here, we compare approximation error in the linear and quadratic
approximations of the parameter-to-objective map, at a fixed control $\vec{\ctrl}^0$ with $z^0_i = 4$ 
for $i = 1, \ldots, \Nc$,
in the sense discussed in section~\ref{sec:trucation_error}. In
Figure~\ref{fig:error_analysis}, we consider the expected value of the
truncation error of the first and second-order Taylor expansion to the
mapping $\ipar \mapsto \QoI(\vec{\ctrl}^0, \ipar)$. This illustrates the rate
of decay of the expected truncation error as indicated by the analysis in
section~\ref{sec:trucation_error}. In particular, the truncation error in a linear
expansion is $\O(\eps)$ and that of the quadratic approximation is
$\O(\eps^{3/2})$. In this study, we consider scaling the distribution law of
$\ipar$ given according to $\mu = \GM{\iparb}{\eps \C}$ with successively
smaller values of $\eps$.

\begin{figure}[t]\centering
\includegraphics[width=.45\textwidth]{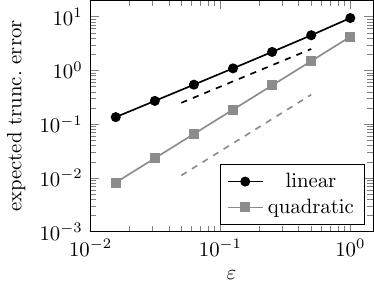}
\caption{The expected values
$\ave\{ |\QoI(\ctrl, \cdot) - \QoI_\mathup{lin}(\ctrl,\cdot)|\}$ (black
solid line) and 
$\ave\{ |\QoI(\ctrl, \cdot) - \QoI_\mathup{quad}(\ctrl,\cdot)|\}$
(gray solid line), with 
distribution law of $\ipar$ given by $\mu = \GM{\iparb}{\eps \C}$ as
$\eps \to 0$. The black and gray dashed lines indicate $\O(\eps)$ and
$\O(\eps^{3/2})$, respectively.
The errors are computed for $\eps_k = 1/2^k$, for $k = 0, \ldots, 6$, using
a fiexed Monte Carlo sample of size $10,000$ to approximate the average errors.  
}
\label{fig:error_analysis}
\end{figure}

\subsection{Results}
\label{label:numerics}
We solve the numerical optimization problem \eqref{equ:OUUpoi} via an
interior point method (to incorporate the box constraints for
$\vec{\ctrl}$) with BFGS Hessian approximation. For that purpose, we
employ MATLAB's 
\verb+fmincon+ routine, to which we provide functions for the
evaluation of the objective function \eqref{equ:OUUpoi} and for the
computation of its gradient with respect to the control, as derived in
the previous section.

We use a uniform control vector $\vec{\ctrl}^0=(4,\ldots,4)$ (see
Figure~\ref{fig:init_control} left) as the initial guess of the optimization
algorithm.  It is instructive to consider the statistical distribution of the
control objective (as defined in~\eqref{equ:controlobj}) for this
initialization, which is shown in the right image in
Figure~\ref{fig:init_control}.  We also depict the distributions
of the linear and quadratic approximations $\QoI_\mathup{lin}(\vec{\ctrl}^0, \ipar)$ and $\QoI_\mathup{quad}(\vec{\ctrl}^0,\ipar)$ as defined in~\eqref{equ:linearapprox}
and~\eqref{equ:quadapprox}, respectively.
For all figures, we have used kernel density estimation (KDE) to
approximate the probability density functions (PDFs) from 10,000
samples.  Note that compared to $\QoI_{\text{lin}}$, the distribution
of $\QoI_{\text{quad}}$ is a significantly better approximation for the
distribution of $\QoI$.
\begin{figure}[ht]\centering
\begin{tabular}{cc}
\includegraphics[width=.4\textwidth]{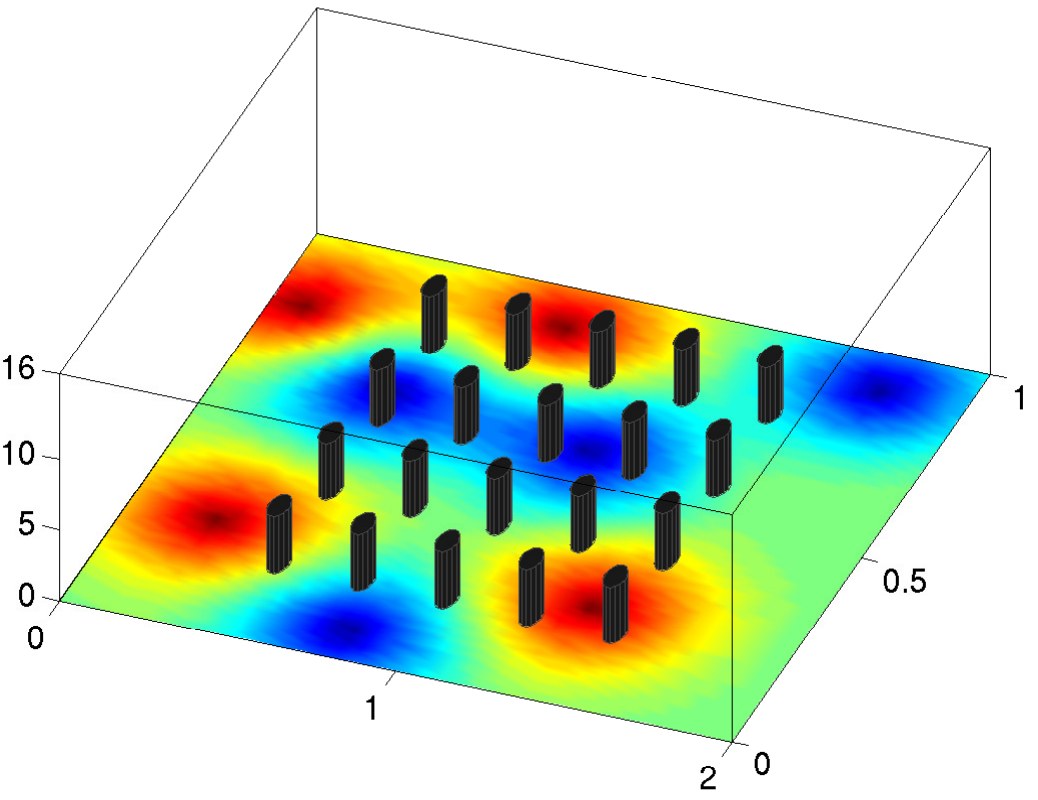}&
\includegraphics[width=.4\textwidth]{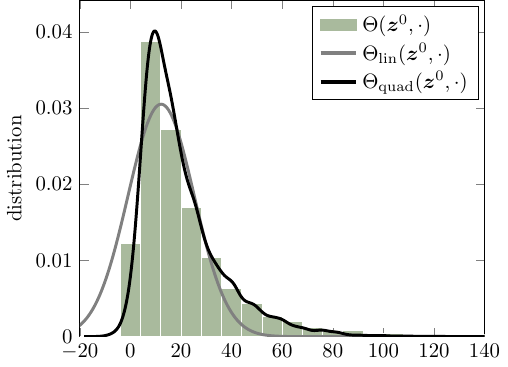}
\end{tabular}
\caption{Left: Control
  $\vec{\ctrl}^0$ with constant values $z_i^0 = 4$, $i \in \{1,
  \ldots, \Nc\}$ (in black). Right: Histogram of $\QoI(\vec{\ctrl}^0,
  \ipar)$ and the distributions of $\QoI_{\mathup{lin}}$
  and $\QoI_{\mathup{quad}}$, which are based on a first and
  second-order expansions of the parameter-to-objective function~\eqref{equ:controlobj}.
}
\label{fig:init_control}
\end{figure}

In Figure~\ref{fig:ouu-quad} we show the risk-averse optimal control
for the risk-aversion parameter
$\beta=1$, the control cost weight $\gamma=10^{-5}$, and where 
we use $\Ntr = 40$ trace estimator vectors for approximating the
traces in the OUU objective function. 
To cope with the nonconvexity of the OUU objective function, we use
a continuation strategy with respect to $\beta$, i.e.,
we solve a sequence of optimization problems with increasing $\beta$,
namely $\{\beta_k\}_{k = 1}^5 = \{0, 0.25, 0.5, 0.75, 1\}$.
The average number of quasi-Newton interior point iterations required
to decrease the residual by $5\times10^{-4}$ was $65$.
The left figure shows the magnitude of the control vector at each injection
well and the right figure depicts the statistical distribution of the
control objective at the computed optimal control for both, $\QoI$ and
$\QoI_{\mathup{quad}}$ (note that the optimal control is based on the
latter). To assess how successful the optimal control is in reducing the
mean and the standard deviation, compare the right images in
Figures~\ref{fig:ouu-quad} and \ref{fig:init_control}.
Notice that compared to the constant control $\vec{\ctrl}^0$, the
optimal control results in a distribution of the objective that is
both shifted to the left (reduction in the mean) and has less spread
(reduction in variance).

\begin{figure}[ht]\centering
\begin{tabular}{cc}
\includegraphics[width=.4\textwidth]{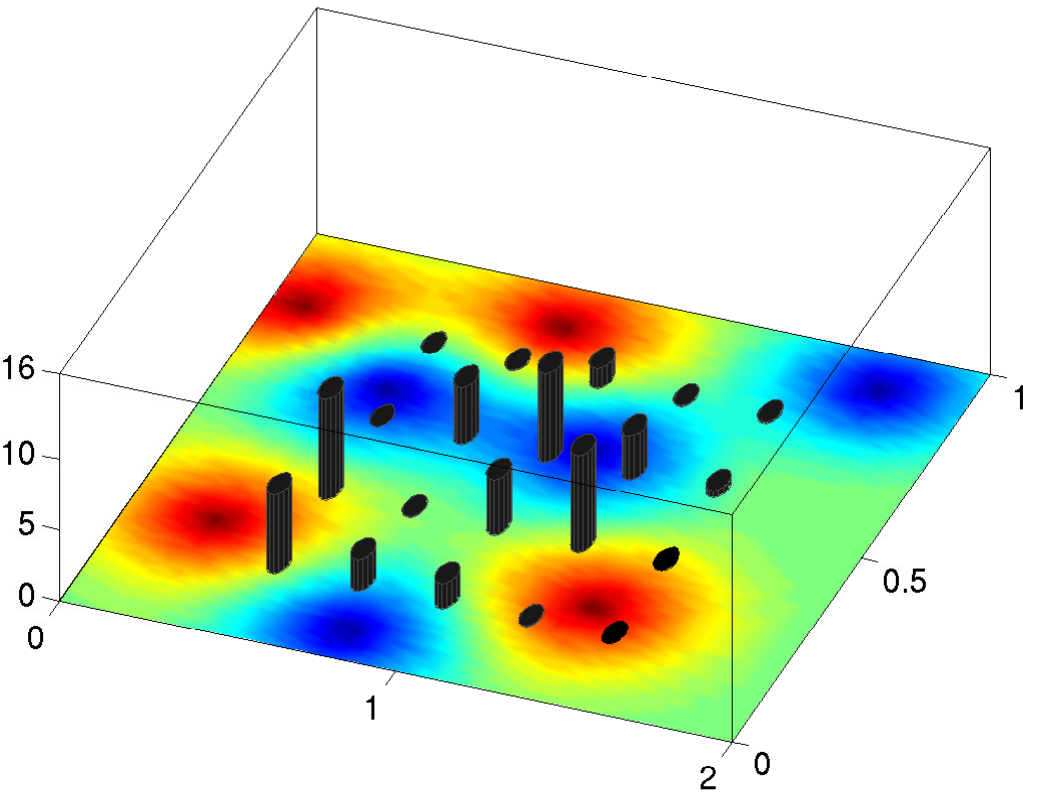}&
\includegraphics[width=.4\textwidth]{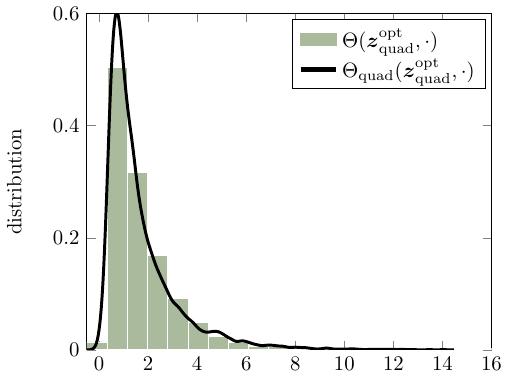}
\end{tabular}
\caption{ Left: The optimal control
  $\vec{\ctrl}^\mathup{opt}_\mathup{quad}$ computed by minimizing the OUU
  objective function with quadratic approximation to the
  parameter-to-objective map~\eqref{equ:controlobj}, and $\beta = 1,
  \gamma=10^{-5}$ and $\Ntr = 40$.  The initial guess for the
  optimizer was  $\vec{\ctrl}^0$ defined in
  Figure~\ref{fig:init_control}.  Right: The distributions of
  $\QoI$ and $\QoI_{\mathup{quad}}$ for the optimal control
  $\vec{\ctrl}^\mathup{opt}_\mathup{quad}$.
}
\label{fig:ouu-quad}
\end{figure}

Next, we study if using the linear (rather than the
quadratic) approximation to the parameter-to-objective map for the
computation of risk-averse optimal control can lead to suboptimal
results. Using $\QoI_\mathup{lin}$ instead of $\QoI_\mathup{quad}$ results in a
simplified version of~\eqref{equ:ouu_cost_quad}, obtained by neglecting
terms involving the Hessian $\Hess$.
We solve the same risk-averse OUU problem as before but with
$\QoI_\mathup{lin}$ rather than $\QoI_\mathup{quad}$. The resulting
optimal control $\vec{\ctrl}_\mathup{lin}^\mathup{opt}$ is shown on the
left of Figure~\ref{fig:ouu-linear},
and the distributions of $\QoI(\vec{\ctrl}_\mathup{lin}^\mathup{opt}, \ipar)$, 
$\QoI_\mathup{lin}(\vec{\ctrl}_\mathup{lin}^\mathup{opt}, \ipar)$ and
$\QoI_\mathup{quad}(\vec{\ctrl}_\mathup{lin}^\mathup{opt}, \ipar)$ are
shown on the right.
Note the large discrepancy between the distributions; in
particular the distribution based on $\QoI_\mathup{lin}$ is a poor
approximation to the actual distribution.
Our goal, however, is to find a control such that the distribution of
the control objective $\QoI$ has small mean and variance. In the
present example, the optimal control computed with the linear
approximation of the parameter-to-objective function $\QoI_\mathup{lin}$
does not perform much worse than the optimal control found with
$\QoI_\mathup{quad}$ in terms of reducing the mean and variance of the
distribution of $\QoI$---despite the poor approximation shown in
Figure~\ref{fig:ouu-linear}. This can be seen from
Figure~\ref{fig:ouu-comapare-linear-quadratic}, where we compare the
distribution of $\QoI$ for the optimal controls
$\vec{\ctrl}_\mathup{lin}^\mathup{opt}$ and
$\vec{\ctrl}_\mathup{quad}^\mathup{opt}$.
We notice that the distribution of
$\QoI(\vec{\ctrl}_\mathup{lin}^\mathup{opt}, \ipar)$ has slightly larger
mean and variance, as can be observed by the thicker tail of the distribution.
\begin{figure}[ht]\centering
\includegraphics[width=.4\textwidth]{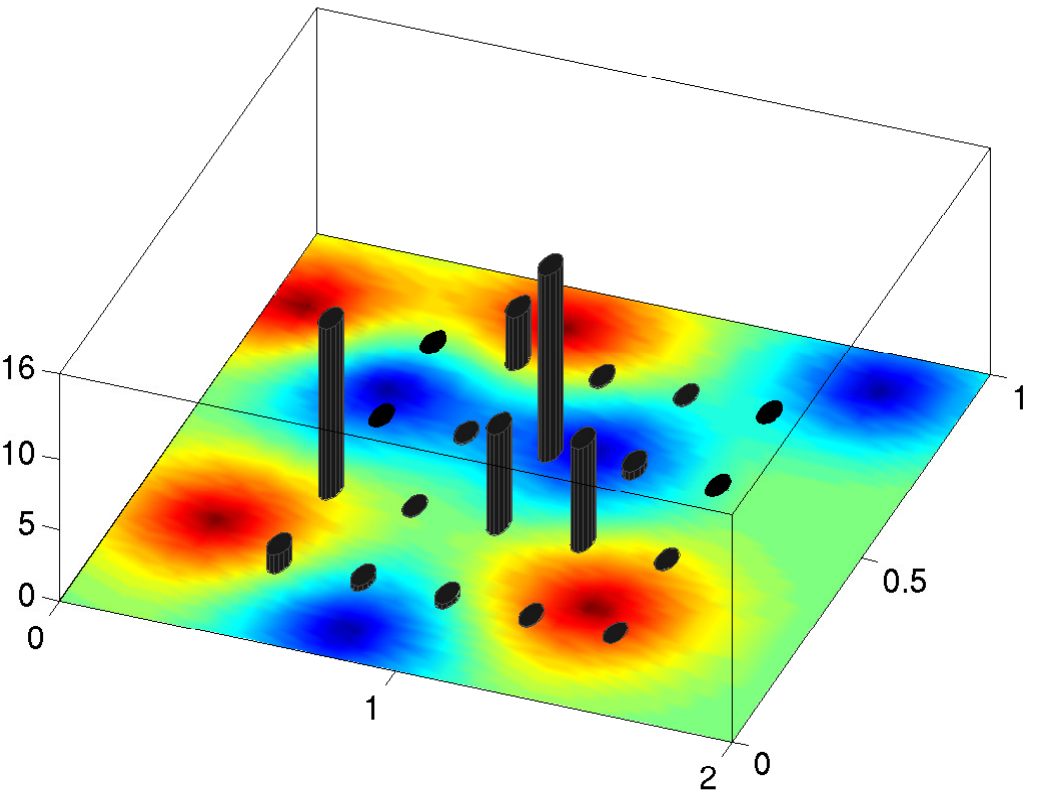}
\includegraphics[width=.4\textwidth]{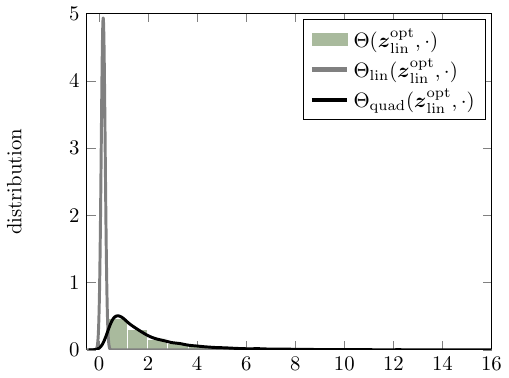}
\caption{Left: Optimal control $\vec{\ctrl}^\mathup{opt}_\mathup{lin}$ computed by minimizing
the OUU objective function with linear approximation to the parameter-to-objective map, and
$\beta = 1, \gamma=10^{-5}$.
Right: Distributions of the control objective and its approximations
for $\vec{\ctrl} = \vec{\ctrl}^\mathup{opt}_\mathup{lin}$.}
\label{fig:ouu-linear}
\end{figure}

\begin{figure}[ht]\centering
\includegraphics[width=0.9\textwidth]{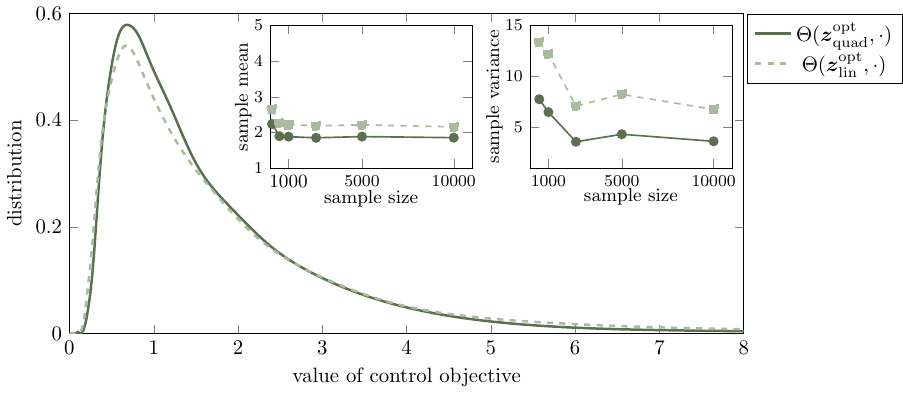}
\caption{Comparison of the distributions of $\QoI(\vec{\ctrl}, \cdot)$
  for $\vec{\ctrl} = \vec{\ctrl}_\mathup{lin}^\mathup{opt}$
  (dashed) and $\vec{\ctrl} = \vec{\ctrl}_\mathup{quad}^\mathup{opt}$ (solid) for $\beta
  = 1$, $\gamma = 10^{-5}$, and $\Ntr = 40$ trace estimation
  vectors.
  The inserts show the Monte Carlo sample convergence for
  the mean and the variance of the distributions.
}
\label{fig:ouu-comapare-linear-quadratic}
\end{figure}

\paragraph{Influence of risk-aversion parameter $\beta$ on the optimal control}
Next, we study the effect of the parameter $\beta$ in 
\eqref{equ:ouu_cost_quad} on the optimal control and the
corresponding distribution of the control objective. In
Figure~\ref{fig:beta_study}, we show results for the mean and the
variance for various risk-aversion parameters $\beta$.  While the
optimal controls have been computed using $\QoI_\mathup{quad}$, we
also report the mean and
variance of $\QoI$.  With increasing
$\beta$ the mean of $\QoI_\mathup{quad}$ increases and the variance of
$\QoI_\mathup{quad}$ decreases, as expected. The optimal controls were computed
using a fixed randomized trace estimator with $\Ntr = 40$.

Next, we consider the effect of increasing $\beta$ on the distribution
of $\QoI$ itself.  While for small values, increasing $\beta$ results
in smaller variances for $\QoI(\vec{\ctrl}, \ipar)$, as desired, for
larger values of $\beta$, the variance increases with $\beta$. This
can be attributed to the fact that for computing the optimal control,
we use the quadratic approximation $\QoI_\mathup{quad}(\vec{\ctrl},
\ipar)$, which only approximates the moments of the control objective
$\QoI(\vec{\ctrl}, \ipar)$.

\begin{figure}[ht]\centering
\includegraphics[width=.8\textwidth]{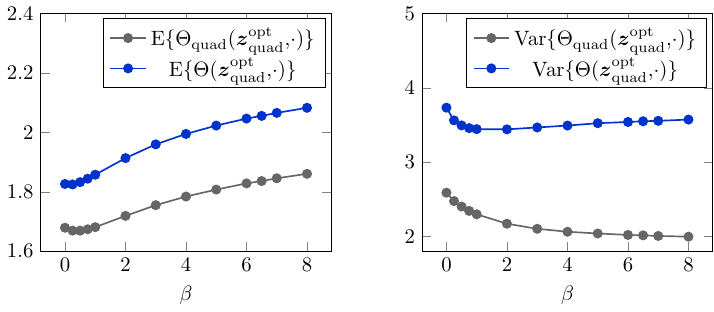}
\caption{Effect of risk aversion parameter $\beta$ on the mean and the
  variance of the control objective~\eqref{equ:controlobj} and its
  quadratic approximation at the optimal control
  $\vec{\ctrl}_\mathup{quad}^\mathup{opt}$. We report statistical values
  for $\QoI_\mathup{quad}$, which has been used in the computation
  of the optimal control, as well as statistical values for the true
  objective $\QoI$. This plots were generated using Monte Carlo
  sampling, with $10,000$ sample points.}
\label{fig:beta_study}
\end{figure}

\paragraph{Comparison between quadratic approximation and Monte Carlo}

Next, we compare the computational cost for computing the optimal
controls using the approach based on the quadratic approximation versus an
SAA approach where Monte Carlo sampling is used to
approximate~\eqref{equ:risk-averse-general}. 
In Figure~\ref{fig:mc_comp} (left), we plot the true OUU objective,
approximated accurately using a large number of samples, against the cost per
iteration measured in number of PDE solves required for objective and gradient
computation,  for three different risk-aversion parameters $\beta$. 
For the
quadratic approximation, we approximate the trace terms in the OUU
objective function using randomized trace estimation (dash-dotted
lines) and using eigenvectors of the covariance-preconditioned Hessian
operator (solid lines), as described in
Section~\ref{sec:ouu_objective}. For the randomized trace estimation we 
use $\Ntr$ values in the range 1--100, and use the same trace estimator
vectors for different values of $\beta$; for the eigenvector-based
approach, we use $\Ntr \in \{1, \ldots, 10\}$. The eigenvectors are computed for a
reference (namely the initial) control. The Monte Carlo sampling
approach is shown with dashed line; for this approach,
we use the same sequence of samples, 
for the different values of $\beta$,
with samples of size $10$, $20$, $40$, $80$, 
$160$, and $320$. Note that in the Monte Carlo approach,
the cost per iteration for objective and gradient computation is 
twice the size of the Monte Carlo sample used.

For a more precise comparison, in
Figure~\ref{fig:mc_comp} (right) for each case we list the data values
corresponding to the largest number of PDE solves. The figure in the
left shows a few important results: (1) the eigenvector-based trace
estimation outperforms the randomized trace estimation for all choices
of $\beta$; in both cases, as the number of vectors used for the
approximation of the trace
increases, the function values level off, indicating convergence 
of the respective optimal controls; (2) 
the optimal controls computed using Monte Carlo sampling of
the objective asymptotically converge to the exact optimal controls of
the true OUU objective; hence, these controls will asymptotically
outperform optimal controls based on the quadratic approximation; and
(3) for low risk-aversion (i.e., small $\beta$) or for limited compute
time (i.e., when only a small number of PDE-solves is afforded), the
quadratic approximation approach is superior to the Monte Carlo SAA
approach.

\begin{figure}\centering
\begin{tikzpicture}
\node (fig) at (-8,0){
\includegraphics[width=.5\textwidth]{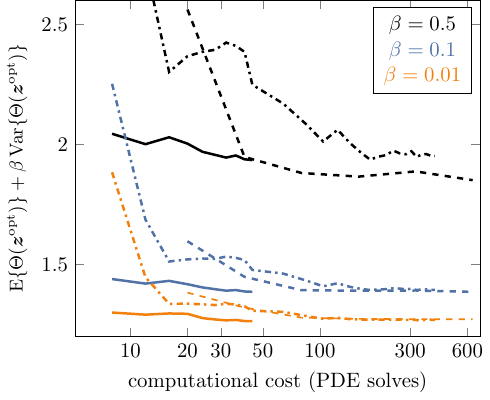}};
\node (tab) at (0,0) {
\begin{tabular}{c|c|c|c}
\hline
   $\beta$ & $\text{MC}$ & $\text{quad}_1$ & $\text{quad}_2$ \\
\hline
   0.5     & 1.852  &  1.935 & 1.952 \\
   0.1     & 1.386  &  1.387 & 1.395 \\
   0.01    & 1.272  &  1.264 & 1.269 \\
   \hline
\end{tabular}};
\end{tikzpicture}
\caption{Left: OUU objective evaluated at differently computed optimal
  controls $\vec{\ctrl}^\mathup{opt}$ versus the number of PDE solves
  required for one objective/gradient computation. Shown are results
  with different risk aversion parameters $\beta$ for optimal controls
  computed using a quadratic approximation (solid and dash-dotted
  lines) and Monte Carlo sampling (dashed lines). For the Monte Carlo
  sampling, the same sequence of samples was used for the different
  values of $\beta$. For the quadratic
  approximation, to estimate the trace we used both the basis of
  eigenvectors of the covariance-preconditioned Hessian, evaluated at
  the initial point $\vec\ctrl_0$ (solid line) and random vectors
  (dash-dotted line).  Right: The OUU objective values corresponding
  to the largest number of PDE solves for various $\beta$ values. The
  columns labeled with MC, quad${_1}$, and quad$_2$ correspond to the
  Monte Carlo sampling, quadratic approximation with eigenvector- and
  random vectors-based trace estimations, respectively. To approximate
  the OUU objective at the computed optimal controls, we use a fixed 
  Monte Carlo sample of size 
  10,000. 
\label{fig:mc_comp}}
\end{figure}

\section{Conclusions}
We propose a scalable method for risk-averse optimal control of systems
governed by PDEs with uncertain parameter fields.
Our approach uses a quadratic approximation of the
parameter-to-objective map, which
enables computing the moments appearing in the OUU objective
function analytically. Moreover, we employ randomized trace
estimators for the operator traces in the OUU objective function. 
The resulting optimization  
problem is constrained by the PDEs defining the gradient $\Grad$ and
the linear action of the Hessian $\Hess$. The resulting method
for risk averse OUU is applicable to
problems with high-dimensional discretized parameter spaces.
This is demonstrated in numerical tests, where we present results for a problem
with a $3{,}200$-dimensional (discretized) parameter space.  Hence, our
approach provides a practical alternative to computationally expensive
sampling-based OUU methods.
The advantages of our method compared to sampling/quadrature methods
are even more pronounced in the context of risk-averse
optimal control of systems governed by \emph{nonlinear} PDEs.  Whereas our
approach requires only one nonlinear PDE solve and about $2\times\Ntr$
linear(ized) PDE solves (to evaluate the OUU objective function), the required
number of nonlinear PDE solves for sampling approaches scales with the number
of samples.

\appendix
\section{Expectation of the truncation error in a Taylor expansion}\label{apdx:taylor_series}
We begin by stating a well-known result~\cite[Theorem
6.6]{Stuart10} regarding Gaussian measures on a Hilbert space. 
\begin{theorem}\label{thm:Gaussian_moments}
Let $\mu = \GM{0}{\C}$ be a Gaussian measure on a Hilbert space $\hilb$. 
For any integer $r \geq 1$, there is a constant $B = B_r \geq 0$ such that, 
\begin{equation}\label{equ:Gaussian_moments}
\int_\hilb \norm{x}^{2r} \, \mu(dx) \leq B \trace(\C)^r.
\end{equation}
\end{theorem}%
Note that the above theorem, for the case of $r = 2$,
follows from~\eqref{equ:bound_forth_moment}.
Next, we define the notation $T [m^p] := T[m, m, \cdots, m]$
for the action of a $p$-linear map $T:\hilb \times \hilb \times \cdots \times \hilb \to \R$.
\begin{lemma}\label{lem:technical_result}
Let $\mu = \GM{0}{\C}$, and 
suppose $T$ is a symmetric $p$-linear map, with $p \geq 2$, such that $|T[m^p]|\leq K \norm{m}^p$ for all $m \in \hilb$. 
Then, we have,
$\int_\hilb |T[m^p]|\, \mu(dm) \leq \tilde{K} \trace(\C)^{p/2}$ for a positive constant $\tilde{K}$. 
\end{lemma}

\emph{Proof.}
For  $p = 2$, the proof is straightforward. 
For $p > 2$, if $p$ is even, i.e., $p = 2k$ for $k \in \Z_+$, then
$
\int_\hilb |T[m^{2k}]|  \mu(dm) \leq K \int_\hilb \norm{m}^{2k} \mu(dm) \leq \tilde{K} \trace(\C)^k, 
$
where the last inequality follows from Theorem~\ref{thm:Gaussian_moments}. 
Note that here $\tilde{K} = KB$ with $B$ from~\eqref{equ:Gaussian_moments}.
In the case $p = 2k+1$ for $k \in Z_+$,
we have,
\[
\begin{aligned}
\int_\hilb &|T[m^{2k+1}]|  \mu(dm)
\leq K \int_\hilb \norm{\ipar}^{2k+1}\,\mu(d\ipar)
\\
&\leq K\Big[\int_\hilb \norm{\ipar}^2 \, \mu(dm)\Big]^{1/2} \Big[\int_\hilb \norm{\ipar}^{4k} \, \mu(dm)\Big]^{1/2}
\leq 
\tilde{K} \trace(\C)^{1/2} \trace(\C)^{k} = \tilde{K} \trace(\C)^{(2k+1)/2}.~\qed
\end{aligned}
\]

\begin{proposition}\label{prp:higher_order_expansions}
Let $\QoI: (\hilb,\borel(\hilb), \mu) \to (\R, \borel(\R))$ be, almost surely, a $p + 1$ times continuously differentiable function, 
and assume $\ipar = \GM{\iparb}{\eps \C}$, with $\eps > 0$.
Suppose $\QoI$ has  $p+1$ uniformly bounded derivatives. Then,
the expected value of the truncation error of the $p$th order Taylor
expansion is
$\O(\eps^{(p+1)/2})$. 
\end{proposition}

\emph{Proof.} 
Since $\QoI:\hilb \to \R$ is $p+1$ times continuously differentiable, 
\[
   \QoI(m) = 
      \QoI(\iparb) + \sum_{n = 1}^p \frac{1}{n!}\QoI^{(n)}(\iparb)(\ipar - \iparb)^n 
      + R_p(\ipar; \iparb). 
\]
The remainder term is given by
$R_p(\ipar; \iparb) = \frac{1}{(p+1)!}\QoI^{(p+1)}(\xi)[(\ipar - \iparb)^{p+1}]$,
where $\xi$ is in the interior of the line segment between $\ipar$ and $\iparb$.
Here we use the notation $\QoI^{(n)}$ for the $n$th derivative.
Now by assumption of the proposition, 
$\QoI^{(k+1)}$ is uniformly bounded on $\hilb$. Therefore, 
by Lemma~\ref{lem:technical_result} we get
$\ave\{ R_p(\ipar; \iparb)\}$
is $\O(\eps^{(p+1)/2})$.~\qed

\section{Derivation of the gradient of the OUU objective function $\J$}
\label{appnd:grad}
Here, we summarize the derivation of the gradient of the OUU objective
function presented in~\eqref{equ:ouu_cost_quad}.  To derive the
expression for the gradient, we employ a formal Lagrangian
approach~\cite{Troltzsch10,BorziSchulz12}, which uses a Lagrangian function composed
of the objective function~\eqref{equ:ouu_cost_quad} with the PDE
constraints~\eqref{equ:risk-averse-ouu-fwd}--\eqref{equ:risk-averse-ouu-inc-adj}
enforced through Lagrange multiplier functions.  This Lagrangian
function $\L$
for the OUU problem is given by:
\begin{equation*}%
\begin{aligned}
  \L&(\ctrl, u, p, \{\incu_j\}, \{\incp_j\}, \ad{u}, \ad{p}, \{\ad{\incu}_j\}, \{\ad{\incp}_j\}) \\
  &= \frac12 \euclidnorm{\Q u - \obsq}^2 + \frac1{2\Ntr} \sum_{j = 1}^\Ntr \ip{\zeta_j}{\psiexpdnj}\\
  & + \frac\beta2\ipbig{\gexpd}{\C[\gexpd]}
  + \frac\beta{4\Ntr}\sum_{j = 1}^\Ntr \norm{\C^{1/2}[\psiexpdnj]}^2\\
  &+ \frac\gamma2 \norm{\ctrl}^2
  + \ip{\Exp{\iparb} \nabla u}{\nabla \ad{u}} - \ip{b + F\ctrl}{\ad{u}}
  + \ip{\Exp{\iparb} \nabla p}{\nabla \ad{p}} + \ip{\Q^*(\Q u - \obsq)}{\ad{p}}\\
  &+ \sum_{j = 1}^\Ntr \Big[\ip{\Exp{\iparb} \nabla \incu_j}{\nabla \ad{\incu}_j} +
     \ip{\zeta_j\Exp{\iparb} \nabla u}{\nabla \ad{\incu}_j}\Big]\\
  &+ \sum_{j = 1}^\Ntr\Big[\ip{\Exp{\iparb} \nabla \incp_j}{\nabla \ad{\incp}_j}
  + \ip{\Q^*\Q \incu_j}{\ad\incp_j} +
  \ip{\zeta_j\Exp{\iparb} \nabla p}{\nabla \ad{\incp}_j}\Big].
\end{aligned}
\end{equation*}
The variables $(u, p, \{\incu_j\}, \{\incp_j\})\in \V^2 \times
(\V^\Ntr)^2$ are the OUU state variables and
$(\ad{u}, \ad{p}, \{\ad\incu_j\}, \{\ad\incp_j\})\in \V^2
\times (\V^\Ntr)^2$ are the OUU adjoint variables,
with  $j \in \{1, \ldots, \Ntr\}$.
Requiring that variations of $\L$ with respect to the OUU adjoint
variables vanish, we recover the OUU state
equations~\eqref{equ:risk-averse-ouu-fwd}-\eqref{equ:risk-averse-ouu-inc-adj}. 
The variations of $\L$ with respect to the OUU state variables are 
\begin{align*}
  \L_u[\ut{u}] &= \ip{\Q^* (\Q u - \obsq)}{\ut{u}}
  + \frac1{2\Ntr} \sum_{j = 1}^\Ntr\ip{\zeta_j}{\expduj}
  + \beta \ip{\Exp{\iparb}\nabla \ut{u}\cdot \nabla p}{\C[ \Grad(\iparb)]}\\
  & + \frac\beta{2\Ntr} \sum_{j = 1}^\Ntr \ip{\Exp{\iparb} (\zeta_j \nabla \ut{u} \cdot \nabla p
    + \nabla \ut{u} \cdot \nabla \incp_j)}{ \C [\psiexpdnj]}\\
  & + \ip{\Exp{\iparb} \nabla \ut{u}}{\nabla \ad{u}}
  + \ip{\Q^*\Q \ut{u}}{\ad{p}}
  + \sum_{j = 1}^\Ntr\ip{\zeta_j\Exp{\iparb} \nabla \ut{u}}{\nabla \ad{\incu}_j},\\
  \L_p[\ut{p}] &= \frac1{2\Ntr} \sum_{j = 1}^\Ntr\ip{\zeta_j}{\expdpj}
  + \beta \ip{\Exp{\iparb}\nabla u\cdot \nabla \ut{p}}{\C[ \Grad(\iparb)]}\\
  & + \frac\beta{2\Ntr} \sum_{j = 1}^\Ntr \ip{\Exp{\iparb} (\zeta_j \nabla u \cdot \nabla \ut{p}
    + \nabla \incu_j \cdot \nabla \ut{p})}{ \C [\psiexpdnj]}\\
  &+ \ip{\Exp{\iparb} \nabla \ut{p}}{\nabla \ad{p}}
  + \sum_{j = 1}^\Ntr \ip{\zeta_j\Exp{\iparb} \nabla \ut{p}}{\nabla \ad{\incp}_j},\\
  \L_{\incu_j}[\ut{\incu}] &=
     \frac1{2\Ntr}\ip{\zeta_j}{\Exp{\iparb} \nabla \ut{\incu} \cdot \nabla p}
  +  \frac\beta{2\Ntr} \ip{\Exp{\iparb}\nabla \ut{\incu} \cdot \nabla p}{\C [\psiexpdnj]}\\
  &+ \ip{\Exp{\iparb} \nabla \ut{\incu}}{\nabla \ad{\incu}_j}
  + \ip{\Q^*\Q \ut{\incu}}{\ad\incp_j},\\
  \L_{\incp_j}[\ut{\incp}] &=
  \frac1{2\Ntr}\ip{\zeta_j}{\Exp{\iparb} \nabla u \cdot \nabla \ut{\incp}}
  + \frac\beta{2\Ntr} \ip{\Exp{\iparb} \nabla u \cdot \nabla \ut{\incp} }{\C [\psiexpdnj]}\\
  &+ \ip{\Exp{\iparb} \nabla \ut{\incp}}{\nabla \ad{\incp}_j},
\end{align*}
with $(\ut{u}, \ut{p}, \ut{\incu}, \ut{\incp}) \in \V^4$.
Letting these variations vanish, results in the
OUU adjoint 
equations~\eqref{equ:ouu-adj-incp-simp-gen}-\eqref{equ:ouu-adj-u-simp-gen}.
Finally, the gradient for~\eqref{equ:ouu_cost_quad} is given by,
\begin{equation}\label{equ:ouu-reduced-gradient-appnd}
   \L_{z}[\ut{z}] := \G(z)\ut{z} = \gamma \ip{z}{\ut{z}} - \ip{F\ut{z}}{\ad{u}}, \quad
   \ut{z} \in L^2(\D).
\end{equation}


\end{document}